\documentclass[leqno,12pt]{article}
\usepackage{latexsym}
\usepackage{amsmath}
\usepackage{amssymb}
\usepackage{enumerate}
\usepackage{graphicx}
\usepackage{tikz}
\usepackage[all]{xy}
\usepackage{placeins}
\usetikzlibrary{decorations.pathreplacing}

\usepackage{theorem}
\newtheorem{theorem}{Theorem}[section]
\newtheorem{proposition}[theorem]{Proposition}
\newtheorem{lemma}[theorem]{Lemma}
\newtheorem{corollary}[theorem]{Corollary}

\theorembodyfont{\rmfamily}
\newtheorem{proof}{\textmd{\textit{Proof.}}}

\newtheorem{remark}[theorem]{Remark}

\newtheorem{definition}[theorem]{Definition}

\makeatletter

\@addtoreset{equation}{section}
\makeatother

\newcommand{\qedd}{\hfill \Box}
\newcommand{\ve}{\varepsilon}

\newcommand{\wt}{\widetilde}

\newcommand{\ol}{\overline}

\newcommand{\B}{\ensuremath{\mathbb{B}}}
\newcommand{\C}{\ensuremath{\mathbb{C}}}

\newcommand{\N}{\ensuremath{\mathbb{N}}}
\newcommand{\R}{\ensuremath{\mathbb{R}}}

\newcommand{\Sph}{\ensuremath{\mathbb{S}}}

\newcommand{\cD}{\ensuremath{\mathcal{D}}}
\newcommand{\cE}{\ensuremath{\mathcal{E}}}

\newcommand{\cL}{\ensuremath{\mathcal{L}}}

\def\Ar{\mathop{\mathrm{Area}}\nolimits}

\def\dist{\mathop{\mathrm{dist}}\nolimits}
\def\Ind{\mathop{\mathrm{Ind}}\nolimits}

\def\Vol{\mathop{\mathrm{Vol}}\nolimits}

\def\Scal{\mathop{\mathit{Sc}}\nolimits}

\def\Int{\mathop{\mathrm{Int}}\nolimits}

\title{
\Large{Curvature at Infinity Governs the Topology 
of Complete Non-Compact Surfaces 
Admitting Schr\"odinger Operators 
of Finite Index}
}

\author{Hideaki HARUMOTO    
and Kei KONDO\footnote{Supported 
by the Grant-in-Aid for Scientific Research (C), 
JSPS KAKENHI Grant Number 22K03288.}
}

\date{\today}

\begin{document}
\maketitle

\begin{abstract}
In this article, we investigate the global topology 
of a complete non-compact Riemannian 
$2$-manifold $\Sigma$ admitting 
a Schr\"odinger operator with non-negative 
potential and finite Morse index. 
While classical results of Fischer-Colbrie 
classify such manifolds under the assumption 
of vanishing index or geometric stability 
as immersed minimal surfaces 
in a Riemannian $3$-manifold, 
we show that the curvature at infinity 
$\lambda_\infty^*(\Sigma)$ 
of the Fischer-Colbrie metric $g^*$---a complete 
conformal metric determined by a positive 
function furnished by Fischer-Colbrie's theorem---governs 
the global topology and geometric rigidity of $\Sigma$ 
without imposing either assumption.

More precisely, 
we derive a fundamental identity relating 
$\lambda_\infty^*(\Sigma)$ 
to the area growth of $(\Sigma,g^*)$, 
show that all critical points of the distance function 
$d_p^*$ from a fixed base point $p$ 
are confined to a bounded region, and, as a corollary, 
obtain a quantitative bound for the number of ends. 
We further distinguish two complementary 
geometric viewpoints. 
On the one hand, a quantitative condition on $\lambda_\infty^*(\Sigma)$ 
forces $\Sigma$ to be diffeomorphic 
to the Euclidean plane $\R^2$. 
On the other hand, when $\Sigma$ has exactly one end, another condition on 
$\lambda_\infty^*(\Sigma)$ guarantees that every Busemann function on $(\Sigma,g^*)$ 
is an exhaustion. 
By clarifying the relationship between these two 
regimes---the critical-point structure of distance 
functions relative to a base point and the global behavior of Busemann functions 
at infinity---we exhibit two complementary manifestations of 
how $\lambda_\infty^*(\Sigma)$ controls the global geometry and topology of $\Sigma$.
\end{abstract}

\section{Introduction}\label{sec1}%%%%%%%%%%%%%%%%%%%%%%%%%%%%%%%%%%%%%%%%%%%%%%%%%%%%%%%%%%%%%%%%%%%%%%%%%%%%%%%%%%%%%%%%%%%%%%%%%%%%%%%%%%%%%%%%%%%%%%%%%%%%%%%%%%%%%%%%%%%%%%%%%%%%%%%%%%%%%%%%%%%%%%%%%%%%%%%%

In her seminal article \cite{F-C1985}, 
Fischer-Colbrie initiated a systematic study 
of Schr\"odinger operators on Riemannian manifolds, motivated by the analysis of the Jacobi operator 
for two-sided minimal surfaces immersed 
in Riemannian $3$-manifolds. 
More precisely, let $(\Sigma,g)$ be a 
(connected) complete non-compact Riemannian 
$2$-manifold, 
where $g$ denotes its Riemannian metric, 
and consider the Schr\"odinger operator 
defined by
\begin{equation}\label{2025_12_01_Schrodinger}
\cL_g:= \Delta_g- K_g +q. 
\end{equation}
Here $\Delta_g$ is the Laplacian operator 
on $C^\infty(\Sigma)$, 
$K_g$ denotes the Gaussian curvature 
of $(\Sigma, g)$, and the potential $q$ belongs 
to $C^\infty(\Sigma) \cap L^\infty(\Sigma)$, 
where $C^\infty(\Sigma)$ and 
$L^\infty(\Sigma)$ denote the spaces 
of smooth functions and essentially 
bounded measurable functions 
on $\Sigma$, respectively. 
If the Morse index of $\cL_g$, denoted 
by $\Ind(\cL_g)$, is finite, 
then she established the following 
fundamental structure theorem:

\begin{theorem}{\rm (\cite{F-C1985})}\label{FCmetric}
If $\Ind(\cL_g) < \infty$, 
then the following assertions hold:
\begin{enumerate}[{\rm (1)}]
\item There exist a compact subset 
$C \subset \Sigma$ and a positive function 
$u \in C^\infty(\Sigma)$ such that 
$\Ind(\cL_g|_{\Sigma \setminus C})=0$ 
and $\cL_g(u) =0$ on $\Sigma \setminus C$. 
\item If, in addition, the potential $q$ is 
non-negative on $\Sigma$, 
then the conformally deformed metric 
$g^*:= u^2g$ satisfies the following properties.
\begin{enumerate}[{\rm (2-1)}]
\item $(\Sigma, g^*)$ is complete and 
its Gaussian curvature $K^*:= K_{g^*}$ 
is non-negative on $\Sigma \setminus C$.
\item $\Sigma$ is finitely connected, 
that is, $\Sigma$ is homeomorphic to 
$N \setminus \{p_1, \dots, p_\ell\}$, 
where $N$ is a compact $2$-manifold and 
$p_1, \dots, p_\ell$ are finitely many points in $N$.
\item The total curvature of $(\Sigma,g^*)$, 
denoted by $c^*(\Sigma)$, satisfies
\[
-\infty<c^*(\Sigma):=
\int_\Sigma K^*\,d{\rm vol}^*
\le2\pi\chi(\Sigma),
\]
and hence is finite. 
Here $d{\rm vol}^*$ denotes the Riemannian 
volume measure associated with $g^*$, 
and $\chi(\Sigma)$ denotes the Euler characteristic 
of $\Sigma$. 
\end{enumerate}
\end{enumerate}
\end{theorem}

In the special case of vanishing index, 
Fischer-Colbrie obtained the following 
complete classification:

\begin{theorem}{\rm (\cite{F-C1985})}\label{FCmetric_ap1}
If $\Ind(\cL_g)=0$ and $q\ge0$, 
then $\Sigma$ is conformally equivalent 
to the complex plane $\mathbb C$, 
is a flat cylinder, or is a flat M\"obius strip. 
In the latter two cases, $q\equiv0$.
\end{theorem}

%%%Part2%%%

\noindent
The geometric significance of this analytic framework 
is illustrated by its application to minimal surface theory. 
Let $(M,\langle\,\cdot\,,\,\cdot\,\rangle)$ 
be a $3$-dimensional Riemannian manifold, 
where $\langle\,\cdot\,,\,\cdot\,\rangle$ 
denotes its Riemannian metric, and 
$
f: (\Sigma, g) 
\to (M, \langle \,\cdot\,, \,\cdot\, \rangle)
$ 
a two-sided minimal immersion with 
$g= f^*\langle \,\cdot\,, \,\cdot\, \rangle$. 
Using the classical Schoen--Yau rearrangement argument, 
we can recast the Jacobi operator $J_\Sigma$ 
of the immersion $f$ in the standard Schr\"odinger form:
\[
J_\Sigma = \Delta_\Sigma - K_\Sigma +\frac{1}{2}\Scal_M \circ f + \frac{1}{2}\|A\|^2,
\]
where $\Scal_M$ denotes the scalar curvature 
of $M$ and $A$ stands for the second 
fundamental form of $(\Sigma, g)$. 
By setting $q := (\Scal_M \circ f + \|A\|^2) / 2$, 
it follows that if $f$ is stable 
($\Ind(J_\Sigma) = 0$) and 
$(M, \langle \,\cdot\,, \,\cdot\, \rangle)$ 
has $\Scal_M \ge 0$, then $q$ is automatically 
non-negative. 
Consequently, Theorem \ref{FCmetric_ap1} immediately implies that $\Sigma$ is conformally equivalent to the complex plane $\C$, 
is a flat cylinder, or is a flat M\"obius strip, 
thereby recovering the classical rigidity 
conclusion within the present 
Schr\"odinger-operator framework.

\bigskip

Throughout this article, 
the conformally deformed metric $g^*$ introduced 
in Theorem \ref{FCmetric} will be referred to 
as the {\em Fischer-Colbrie metric}. 
A brief historical remark concerning this metric is in order. 
Originally, Fischer-Colbrie assumed $\Sigma$ 
to be {\em oriented} when deriving 
items (2-2) and (2-3) of Theorem \ref{FCmetric}, 
in order to apply Huber's classical result \cite{Hub}. 
However, this orientation assumption is in fact superfluous, as Huber's theorem holds 
in full generality without it. 
For a comprehensive and modern treatment of Huber's theorem and related topics, 
we refer the reader to the celebrated 
monograph \cite{SST} by Shiohama, Shioya, 
and Tanaka.

Under the notation and assumptions 
in Theorem \ref{FCmetric}, 
$\chi(\Sigma)$ is given by
\[
\chi(\Sigma)=\chi(N)-\ell.
\]
Thus, item (2-2) of Theorem \ref{FCmetric} shows 
that $\Sigma$ has finite topology and, 
in particular, finite Euler characteristic, 
while item (2-3) guarantees that its total 
curvature is finite. 
This finite-topology property will later be related 
to the critical-point theory of distance 
functions developed by Grove and 
Shiohama \cite{GS} 
(see item (4) of Theorem \ref{2025_11_27_thm1.4}). 
These classical results naturally raise 
the following question:
\begin{center}
{\em Can one control the topology, under the 
finite-index assumption, by suitable curvature conditions at infinity without imposing stability 
or index zero?}
\end{center}
The classical results of Fischer-Colbrie control 
the underlying topology through vanishing index 
or the geometric stability of immersed 
minimal surfaces in a Riemannian $3$-manifold. 

%%%Part3%%%

\bigskip

In sharp contrast to these classical developments, 
the main objective of this article is 
{\em to demonstrate that the curvature at 
infinity governs the topology of complete 
non-compact surfaces without assuming 
either vanishing index or the existence 
of a stable two-sided minimal immersion 
of $(\Sigma,g)$ into a Riemannian $3$-manifold}.

\bigskip

Before stating our main theorems,
we clarify the standing assumptions and notation. 
Throughout the remainder of this article, 
we assume that the complete non-compact 
Riemannian $2$-manifold $(\Sigma, g)$ satisfies 
\begin{equation}\label{2025_12_01_whole_assumption}
\text{$\Ind(\cL_g) < \infty$ \quad and \quad $q \ge 0$.}
\end{equation}
Under Eq.\,\eqref{2025_12_01_whole_assumption}, 
we fix a compact set $C$ and 
a positive function $u$ furnished 
by Theorem \ref{FCmetric}~(1). 
Enlarging $C$ if necessary, 
we may and do assume that 
$\Int C\neq\emptyset$. 
We then let $g^*=u^2g$ be the corresponding 
Fischer-Colbrie metric on $\Sigma$. 
We denote by $d^*$ the distance induced by $g^*$. 
For any $x \in \Sigma$ and $r>0$, 
$B^*_r(x) := \{y \in \Sigma \mid d^*(x, y)< r\}$ 
denotes the open metric ball in $(\Sigma, g^*)$, 
whose area is defined by 
$
\Ar (B^*_r(x)) := \int_{B^*_r(x)}d{\rm vol}^*
$. 
Furthermore, $K^*$ denotes, 
as in Theorem \ref{FCmetric}~(2-1), 
the Gaussian curvature of $(\Sigma,g^*)$, 
and the curvature at infinity $\lambda_\infty^*(\Sigma)$ of $(\Sigma,g^*)$ 
is defined by
\begin{equation}\label{curvature_at_infinity}
\lambda_\infty^*(\Sigma)
:=\lambda_\infty(\Sigma,g^*)
:=2\pi\chi(\Sigma)-c^*(\Sigma),
\end{equation}
where $c^*(\Sigma)$ denotes the 
total curvature of $(\Sigma,g^*)$. 
Note that by Cohn-Vossen's theorem \cite[Satz 6]{CV1}, $\lambda_\infty^*(\Sigma)\ge 0$. 

\bigskip

Our first main result is the following:

%%%Part4%%%

\begin{theorem}{\rm (Fundamental Structure Theorem)}\label{2025_11_27_thm1.4} 
For any fixed $p \in \Int C$, 
there exists a non-compact model surface 
of revolution $\wt{\Sigma}$ 
with base point $\tilde{p}$, 
whose radial curvature function 
$\wt K$ is non-positive on $[0,\infty)$ 
and compactly supported, such that 
\begin{enumerate}[{\rm (1)}]
\item the radial curvature of $(\Sigma, g^*)$ 
at $p$ is bounded from below by $\wt{K}$, 
\item there exists a constant $\alpha^*\in[1/2,\infty)$, depending on $g^*$ and $\tilde g$, 
such that 
\begin{equation}\label{2025_11_27_thm1.4_eq1}
\lambda^*_\infty(\Sigma) 
= 
4 \alpha^* \pi \lim_{t\to \infty}
\frac{\Ar (B_t^* (p))}{\Ar (B_t (\tilde{p}))}
\end{equation}
holds, where $\tilde{g}$ denotes 
the Riemannian metric of $\wt{\Sigma}$, 
and $B_t (\tilde{p})$ denotes the open metric 
ball on $\wt{\Sigma}$ centered at $\tilde{p}$ 
with radius $t>0$, 
\item the total curvature $c(\wt{\Sigma})$ 
of $\wt{\Sigma}$ is finite and is given by 
\[
c(\wt{\Sigma}) = 2\pi (1 -2 \alpha^*),
\] 
\item there exists a radius $R>0$ 
such that $\ol{B^*_R(p)}$ contains both 
the compact set $C$ and all critical points 
of the distance function 
$d_p^*(\,\cdot\,) := d^*(p,\, \cdot\,)$ 
in the sense of Grove and Shiohama.
\end{enumerate}
\end{theorem}

\medskip

The fundamental structure theorem also leads 
to the following estimate for the number of ends, 
whose proof requires a separate 
comparison-geometric argument.

\begin{corollary}\label{2026_08_02_end_estimate}
Let $\alpha^*$ be the constant furnished 
by Theorem~\ref{2025_11_27_thm1.4}. 
Then the number of ends of $\Sigma$ 
is at most $4\alpha^*$.
\end{corollary}

\begin{remark}
We make two remarks concerning 
Theorem \ref{2025_11_27_thm1.4}: 
\begin{enumerate}[{\rm (1)}]
\item The basic framework 
of radial curvature geometry---including 
the precise definitions of a non-compact 
model surface of revolution, radial curvature bounds, 
and related geometric concepts---will be reviewed 
in Section \ref{pre}. 
\item A point $q \in \Sigma \setminus \{p\}$ is said to be a 
{\em 
critical point of $d_p^*$ in the sense 
of Grove--Shiohama
}
\cite{GS} if for each tangent 
vector $v \in T_q \Sigma \setminus \{o_q\}$, 
there exists a unit-speed minimizing 
geodesic segment 
$
\gamma :[0,d_p^*(q)]\to(\Sigma,g^*)
$ 
emanating from $p=\gamma(0)$ 
to $q=\gamma(d_p^*(q))$ such that
\[
\angle^*
\Big(
-\gamma'(d_p^*(q)),v
\Big)
\le\frac{\pi}{2},
\]
where $\angle^*(u,v)$ denotes the angle 
between two non-zero vectors 
$u,v\in T_q\Sigma$. 
By convention, the base point $p$ itself is 
also designated as a critical point of $d_p^*$.
\end{enumerate}
\end{remark}

From this critical-point-theoretic perspective, 
Theorem \ref{2025_11_27_thm1.4}~(4) 
provides a critical-point-theoretic counterpart 
to the finite topological type of $\Sigma$: 
all non-trivial critical points of $d_p^*$ 
are confined within the compact metric ball 
$\ol{B_R^*(p)}$. 
In this light, the classical hypotheses discussed 
above---such as $\Ind(\cL_g)=0$ 
or the geometric stability of two-sided 
minimal surfaces immersed in a Riemannian 
$3$-manifold with non-negative scalar 
curvature---may be viewed, 
from the perspective developed here, 
as strong analytic conditions leading 
to particularly restrictive global topologies.

In what follows, for each $x \in \Sigma$, 
let $A_x^*$ denote the set of all unit vectors 
tangent to rays of $(\Sigma, g^*)$ emanating 
from $x$, and let $\mu$ denote 
the Lebesgue measure on the unit circle 
$(\Sph^1_x)^*:= \{v \in T_x\Sigma \mid \|v\|^*=1 \}$ 
with respect to the norm $\|v\|^*:=\sqrt{g^*_x (v, v)}$ for all $v \in T_x\Sigma$. 
Moreover, let $\tilde{g}$ denote 
the Riemannian metric of $\wt\Sigma$, 
and let $\wt{K}$ be its non-positive radial 
curvature function. 
Theorem \ref{2025_11_27_thm1.4}, 
together with Theorem \ref{2025_12_13_thm2.7} in Section \ref{pre}, 
leads to the following rigidity result: 
the curvature-at-infinity condition below eliminates 
all non-trivial critical points of $d_p^*$ 
and, consequently, forces $\Sigma$ to 
be diffeomorphic to $\R^2$.

\begin{theorem}\label{2025_11_27_thm1.6}
Let $(\wt\Sigma,\tilde g)$ and $\alpha^*$ be as 
in Theorem~\ref{2025_11_27_thm1.4}, 
and let $\wt K$ denote the radial curvature function of $\wt\Sigma$. 
If
\begin{equation}\label{2025_11_27_thm1.6_eq1}
\lambda^*_\infty(\Sigma) 
\ge 
2\alpha^* \pi 
\Big\{
2 - \exp \Big(\int_0^\infty t \wt{K}(t)dt \Big)
\Big\}, 
\end{equation}
then $d_p^*$ has no critical points 
on $\Sigma \setminus \{p\}$ 
in the sense of Grove--Shiohama. 
In particular, $\Sigma$ is diffeomorphic to $\R^2$. 
Moreover, setting $K^*_+:=\max\{0, K^*\}$, 
we have the estimate 
\begin{equation}\label{2025_11_27_thm1.6_eq2}
2\pi - \int_{\Sigma} K^*_+d{\rm vol}^* 
\le 
\inf_{x \in \Sigma} \mu (A^*_x) 
\le 
2\pi -c^*(\Sigma). 
\end{equation}
\end{theorem}

\medskip\noindent
The relation between
Theorems~\ref{2025_11_27_thm1.4}
and~\ref{2025_11_27_thm1.6}
is illustrated schematically in
Figure~\ref{fig:rigidity}.

\begin{figure}[htbp]
\centering
\begin{tikzpicture}[scale=1, >=stealth]

% ==========================================
% Left Panel: General Case (Theorem 1.4)
% ==========================================

\begin{scope}[xshift=-3.5cm, yshift=-0.95cm]
% Base point 
\fill (0,0) circle (1.5pt) node[below, yshift= 0.7mm] {$p$};

% Small level curves near p
\draw[blue, thick] (0,0) circle (0.4);
\draw[blue, thick] (0,0) circle (0.8);

% Schematic critical-point structure
\draw[red, thick]
(0,1.3) .. controls (1.5,0.5) and (1.0,-0.8) .. 
(0,-0.8) .. controls (-1.0,-0.8) and (-1.5,0.5) .. (0,1.3);
\draw[red, thick]
(0,1.3) .. controls (1.0,2.0) and (0.8,2.7) .. (0,2.7) 
.. controls (-0.8,2.7) and (-1.0,2.0) .. (0,1.3);

% Possible critical point
\fill[red] (0,1.3) circle (1.5pt);
% Pointer line and label
\draw[red, thin, <-] (0.09,1.3) -- (1.5,2.0)
node[right, text=red] 
{\footnotesize possible critical point};

% Compact region containing all possible critical points
\draw[blue, thick] (0,0.95) ellipse (1.6 and 2.15);
\end{scope}

% Text Label at the bottom for Left Panel
\node[align=center] at (-3.5,-3.2)
{\textbf{General Case}\\
\footnotesize (Theorem \ref{2025_11_27_thm1.4})\\
\footnotesize Critical points, if any, are confined};

% ==========================================
% Right Panel: Rigidity Case (Theorem 1.6)
% ==========================================
\begin{scope}[xshift=3.5cm]
% Base point
\fill (0,0) circle (1.5pt) node[below, yshift= 0.75mm] {$p$};

% Nearly circular but slightly distorted nested level curves
\draw[blue, thick, smooth cycle, tension=0.95]
plot coordinates {
(0.00,0.40)
(0.30,0.24)
(0.33,-0.05)
(0.14,-0.30)
(-0.12,-0.34)
(-0.34,-0.12)
(-0.28,0.20)
};

\draw[blue, thick, smooth cycle, tension=0.95]
plot coordinates {
(0.00,0.80)
(0.55,0.48)
(0.68,0.02)
(0.42,-0.52)
(-0.06,-0.72)
(-0.56,-0.50)
(-0.70,-0.02)
(-0.42,0.58)
};
\draw[blue, thick, smooth cycle, tension=0.95]
plot coordinates {
(0.00,1.30)
(0.82,0.88)
(1.08,0.18)
(0.90,-0.58)
(0.22,-1.06)
(-0.58,-1.02)
(-1.02,-0.42)
(-0.92,0.52)
(-0.38,1.10)
};
\draw[blue, thick, smooth cycle, tension=0.95]
plot coordinates {
(0.00,1.70)
(1.02,1.18)
(1.38,0.26)
(1.20,-0.72)
(0.42,-1.38)
(-0.52,-1.42)
(-1.26,-0.86)
(-1.36,0.18)
(-0.86,1.22)
};
\draw[blue, thick, smooth cycle, tension=0.95]
plot coordinates {
(0.00,2.15)
(1.22,1.52)
(1.72,0.42)
(1.58,-0.78)
(0.82,-1.72)
(-0.28,-1.96)
(-1.28,-1.48)
(-1.82,-0.42)
(-1.62,0.88)
(-0.82,1.82)
};
\end{scope}

% Text Label at the bottom for Right Panel
\node[align=center] at (3.5,-3.2)
{\textbf{Under the Rigidity Condition}\\
\footnotesize (Theorem \ref{2025_11_27_thm1.6})\\
\footnotesize No non-trivial critical points\\
\footnotesize $\Sigma \cong \R^2$};
\end{tikzpicture}

\caption{
A purely schematic illustration of the transition 
from Theorem \ref{2025_11_27_thm1.4} 
to Theorem \ref{2025_11_27_thm1.6}. 
Theorem \ref{2025_11_27_thm1.4} confines 
all non-trivial critical points of $d_p^*$ to a 
compact region, whereas, 
under the curvature-at-infinity condition 
in Theorem \ref{2025_11_27_thm1.6}, 
no non-trivial critical points remain, 
and consequently $\Sigma$ is diffeomorphic 
to $\R^2$.
}
\label{fig:rigidity}
\end{figure}
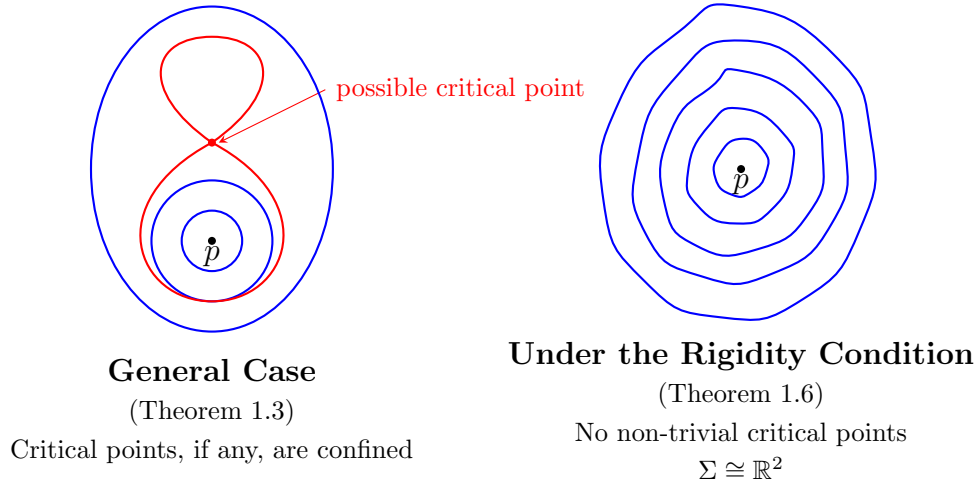
\FloatBarrier

%%%Part5%%%

\begin{remark}\label{2026_08_08_rem_new1}
We place Theorem \ref{2025_11_27_thm1.6} 
in the broader context of recent developments 
in minimal hypersurface theory. 
Let
\[
f:(M^n,h)\to
(\R^{n+1},\langle\,\cdot\,,\cdot\,\rangle)
\]
be a complete two-sided minimal hypersurface 
with the induced metric $h=f^*\langle\,\cdot\,,\cdot\,\rangle$. 
For $n=2$, do Carmo and Peng \cite{dCP}, 
and independently Fischer-Colbrie and Schoen \cite{F-CS}, 
proved that stability forces $f(M)$ to be a hyperplane. 
The same rigidity has recently been established 
for $n=3$ by Chodosh and Li \cite{CL}, 
with another proof by Catino, Mastrolia, and Roncoroni 
\cite{CMR}; for $n=4$ by Chodosh, Li, Minter, and 
Stryker \cite{CLMS}; and for $n=5$ by Mazet \cite{Maz}. 
These developments highlight the close relationship between 
stable elliptic operators and global geometry. 
While these higher-dimensional results concern ambient rigidity, 
Theorem \ref{2025_11_27_thm1.6} instead yields intrinsic 
topological control under general finite-index constraints.
\end{remark}

While Theorem \ref{2025_11_27_thm1.6} 
gives a global rigidity result by eliminating 
all non-trivial critical points of the distance 
function $d_p^*$, a different geometric 
phenomenon arises when one turns 
to Busemann functions. 
To capture the geometry at infinity more directly, 
it is natural to replace the distance function 
from a fixed base point by a Busemann 
function associated with a geodesic ray.

The fundamental geometric distinction is 
that $d_p^*$ describes the geometry relative 
to the fixed base point $p$, 
whereas a Busemann function reflects 
the geometry along a designated geodesic ray. 
Even when all non-trivial critical points 
of $d_p^*$ are confined within a bounded region, 
as guaranteed by 
Theorem \ref{2025_11_27_thm1.4}~(4), 
a Busemann function on $(\Sigma,g^*)$ 
need not be an exhaustion. 
Without sufficient geometric control at infinity, 
its level sets need not be compact.

The following theorem provides a separate 
geometric condition that controls the global 
behavior of Busemann functions 
near infinity and, in particular, guarantees 
their exhaustion property. 
Thus, Theorem \ref{2025_11_27_thm1.6} 
and the theorem below describe two 
complementary manifestations 
of geometric control: the disappearance 
of non-trivial critical points of $d_p^*$ relative 
to a fixed base point, and the global exhaustion behavior 
of Busemann functions associated with geodesic rays. 

\begin{theorem}\label{2025_12_01_thm1.7}
Let $(\wt\Sigma,\tilde g)$ and $\alpha^*$ 
be as in Theorem~\ref{2025_11_27_thm1.4}. 
Assume that $\Sigma$ has exactly one end 
and that $\lambda^*_\infty(\Sigma)>0$. 
Define
\[
\beta^*
:=
\alpha^*
\lim_{t\to\infty}
\frac{\Ar(B_t^*(p))}
{\Ar(B_t(\tilde p))}.
\]
By Theorem~\ref{2025_11_27_thm1.4}~(2),
\[
\lambda_\infty^*(\Sigma)=4\pi\beta^*.
\]
Then $\beta^*>0$, and the following assertions hold:
\begin{enumerate}[{\rm (1)}]
\item If $\beta^*\in(0,1/4)$, 
then every Busemann function on $(\Sigma,g^*)$ is 
an exhaustion. 
Moreover, for each $\ve >0$, 
there exists a compact subset $W_\ve$ 
of $\Sigma$ such that 
\begin{equation}\label{2026_07_31_thm1.7_new1}
|\mu (A_x^*) - \lambda^*_\infty (\Sigma)| < \ve
\end{equation}
for all $x \in \Sigma \setminus W_\ve$. 
Furthermore, for an increasing exhaustion 
$\{V_i\}_{i\in\N}$ of $\Sigma$ by compact subsets, we have 
\begin{equation}\label{2026_07_31_thm1.7_new2}
\lim_{i\to\infty}
\frac{
\int_{V_i}\mu (A^*_x)\,d{\rm vol}^*
}
{
\int_{V_i}\,d{\rm vol}^*
}
=
\lambda_\infty^*(\Sigma). 
\end{equation}
\item If $\beta^*\in (0,\chi(\Sigma) / 2- 1/4)$,
then $\Sigma$ is homeomorphic to $\R^2$. 
In particular, 
Eq.\,\eqref{2025_11_27_thm1.6_eq2} holds. 
Moreover, every Busemann function 
on $(\Sigma,g^*)$ is an exhaustion.
\end{enumerate}
\end{theorem}

\begin{remark}
For an extension of the properness and 
topological exhaustion properties of 
Busemann functions to higher dimensions 
from the perspective of radial curvature geometry, 
we refer the reader to the work of the second 
author (\textbf{K.}) and Tanaka \cite{KT_MA}.
\end{remark}

Taken together, the results of this article 
provide a unified perspective on how curvature 
at infinity governs the global structure of 
complete non-compact surfaces.

%%%Part6%%%

\medskip

The remainder of the article is organized as follows. 
In Section \ref{pre}, we review the basic framework 
of radial curvature geometry, establishing 
the necessary tools and comparison criteria. 
Section \ref{sec3} contains the proofs of all 
the main results. 
We establish, in succession, 
the fundamental structure 
theorem (Theorem \ref{2025_11_27_thm1.4}), 
the quantitative estimate on the number 
of ends (Corollary \ref{2026_08_02_end_estimate}) 
as a consequence 
of Theorem \ref{2025_11_27_thm1.4}, 
the critical-point elimination and 
global rigidity theorem 
(Theorem \ref{2025_11_27_thm1.6}), 
and the global exhaustion properties 
of Busemann functions 
(Theorem \ref{2025_12_01_thm1.7}).

\bigskip

This article is based 
in part on the master's thesis \cite{Har} 
of the first author (\textbf{H.}) 
at Okayama University.

%%%%Section 2%%%%%

\section{Preliminaries}\label{pre}

In this section, we collect the principal geometric tools 
that will be used throughout the proofs of the main results. 
We first review the fundamental notions 
of radial curvature geometry for pointed 
complete non-compact Riemannian manifolds, 
which provide the comparison 
framework for the critical-point and 
end-structure analysis of the Fischer-Colbrie 
metric carried out in the next section. 
We then recall several results on complete 
non-compact surfaces that relate total 
curvature and curvature at infinity to 
the geometry of rays and Busemann functions. 
Together, these two groups of results provide 
the geometric foundation for the proofs 
of the main theorems in Section \ref{sec3}.

\subsection{Radial curvature geometry}\label{sub2.1}

Let $\wt{M}$ be a complete non-compact 
Riemannian $2$-manifold homeomorphic to $\R^2$, 
let $\tilde{p}\in\wt M$ be a base point, and set 
$
\Sph^{1}_{\tilde{p}} 
:= \{ v \in T_{\tilde{p}} \wt{M} \mid \| v \| = 1 \}
$.

\begin{definition}{\rm (Model Surfaces of Revolution)}
\begin{enumerate}[{\rm (i)}]
\item We call the pair $(\wt{M}, \tilde{p})$ 
a {\em non-compact model surface of revolution} 
if its Riemannian metric $\tilde{g}$ is expressed 
in terms of geodesic polar coordinates 
around $\tilde{p}$ as 
\begin{equation}\label{polar}
\tilde{g} = dt^2 + f(t)^2d \theta^2
\end{equation}
for all 
$(t,\theta) \in (0,\infty) \times {\Sph_{\tilde{p}}^1}$, where $f:(0,\infty)\to\R$ is a positive smooth function
that extends to a smooth odd function near $0$ 
and satisfies the differential equation 
\[
f''(t) + G (\tilde{\gamma}(t)) f(t) = 0
\]
with the initial conditions $f(0) = 0$ and $f'(0) = 1$.  
\item The function 
$G \circ \tilde{\gamma} : [0,\infty) \to \R$ 
is called the {\em radial curvature function 
of $\wt{M}$}, 
where $G$ denotes the Gaussian curvature 
of $\wt{M}$, 
and $\tilde{\gamma}$ is any meridian 
emanating from $\tilde{p} = \tilde{\gamma} (0)$. 
In what follows, we set 
\[
\wt{K} := G \circ \tilde{\gamma}
\]
on $[0,\infty)$. 
\item The $m$-dimensional model manifolds of revolution are defined analogously. 
\end{enumerate}
\end{definition}

\begin{remark}\label{2026_05_24_rem1}
We note that a complete classification 
of such $n$-dimensional model manifolds 
of revolution was established by Katz and 
the second author (\textbf{K.}); 
for a detailed exposition of this 
structural classification, we refer the reader 
to \cite{KK}.
\end{remark}

Let $(\wt{M}, \tilde{p})$ be a non-compact 
model surface of revolution, and let $\wt{K}$ 
be the radial curvature function of $\wt{M}$. 

\begin{definition}{\rm (Radial Curvature Bounded Below)} 
Given a complete non-compact Riemannian 
$m$-manifold $M$ with a base point $p \in M$, 
we say that the pair $(M, p)$ 
has {\em radial curvature at $p$ bounded from 
below by $\wt{K}$} if along every 
unit-speed minimizing geodesic 
$\gamma: [0,a) \to M$ emanating 
from $p = \gamma (0)$, 
its sectional curvature $K_M$ satisfies
\[
K_M(\sigma_{t}) \ge \wt{K}(t)
\]
for all $t \in [0, a)$ 
and all $2$-dimensional linear 
subspaces 
$\sigma_t\subset T_{\gamma(t)}M$ 
spanned by $\gamma'(t)$ 
and a vector $v\perp\gamma'(t)$.
\end{definition}

\begin{remark}\label{2026_05_24_rem2}
For instance, if the metric $\tilde{g}$ of $\wt{M}$ 
is given by $dt^2 + t^{2}d \theta^2$ or 
$dt^2 + \sinh^{2} t\,d \theta^2$, 
then the radial curvature function satisfies 
$\wt{K}(t) = 0$ or $\wt{K}(t) = -1$, respectively. 
In general, the radial curvature function $\wt{K}$ 
may change sign. Indeed, 
as shown in \cite[Theorems 1.4 and 3.1]{KT5}, 
there exist non-compact model surfaces 
of revolution with finite total curvature 
whose radial curvature functions satisfy
\[
\liminf_{t\to\infty}\wt K(t)=-\infty
\quad\text{or}\quad
\limsup_{t\to\infty}\wt K(t)=\infty.
\]
Consequently, employing a model surface 
of revolution provides a significantly wider 
framework than the conventional 
comparison geometry based on space forms 
of constant curvature.
\end{remark}

To analyze the global behavior and critical points 
of distance functions under radial curvature bounds, we use a generalized Toponogov comparison 
theorem that extends the classical comparison 
with space forms to a broader class of 
model surfaces. The comparison theorem below 
will be applied in the proof 
of Corollary \ref{2026_08_02_end_estimate}, 
where the absence of a pair of cut points 
in the corresponding model sector plays 
a decisive role.

\begin{theorem}{\rm (A New Type of Toponogov Comparison Theorem \cite{KT1})}\label{new_TCT}
Let $(M, p)$ be a connected complete 
non-compact Riemannian manifold whose 
radial curvature at $p \in M$ is bounded 
from below by $\wt{K}$ of $(\wt{M}, \tilde{p})$. 
If the set 
$
\wt{V}(\delta_{0})
:=
\{ \tilde{x} \in \wt{M} \setminus \{\tilde{p}\} 
\mid 
0 < \theta(\tilde{x}) < \delta_0\}
$ 
has no pair of cut points for some 
$\delta_{0} \in (0, \pi]$, 
then for any geodesic triangle $\triangle(pxy)$ 
in $M$ with $\angle (xpy) < \delta_{0}$, 
there exists a geodesic triangle 
$
\wt{\triangle} (pxy) 
:=\triangle(\tilde{p}\tilde{x}\tilde{y})
$ in $\wt{V}(\delta_{0})$ such that
\begin{equation}\label{TCT_length}
\tilde{d}(\tilde{p},\tilde{x})=d(p,x), \quad 
\tilde{d} (\tilde{p},\tilde{y})=d(p,y), \quad 
\tilde{d} (\tilde{x},\tilde{y})=d(x,y) 
\end{equation}
and
\begin{equation}\label{TCT_angle}
\angle (xpy) \ge \angle (\tilde{x}\tilde{p}\tilde{y}), \quad \angle (pxy) \ge \angle (\tilde{p}\tilde{x}\tilde{y}), \quad \angle (pyx) \ge \angle (\tilde{p}\tilde{y}\tilde{x}), 
\end{equation}
where $d$ (resp.\ $\tilde{d}$) denotes the distance function of $M$ (resp.\ $\wt{M}$), 
and $\angle(pxy)$ denotes the angle 
between the minimizing geodesics from $x$ 
to $p$ and from $x$ to $y$ forming the triangle 
$\triangle(pxy)$.
\end{theorem}

\medskip\noindent
The comparison in Theorem~\ref{new_TCT} 
is illustrated schematically in Figure~\ref{fig:toponogov}. 

\begin{figure}[htbp]
\centering
\begin{tikzpicture}[scale=1.5, >=stealth]
        
%========================================% Left Panel: M%========================================

\begin{scope}[xshift=-3.2cm]
\coordinate (P) at (0,0);
\coordinate (A) at (-1.2, 2);
\coordinate (B) at (1.2, 2);

\node[anchor=south west] at (-1.2, 2.5)
{$M \supset \triangle$};
\draw[thick] (P) to[bend left=15] node[midway, left=3pt] {$r_1$} (A);
\draw[thick] (P) to[bend right=15] node[midway, right=3pt] {$r_2$} (B);
\draw[thick] (A) to[bend left=15] node[midway, above=3pt] {$s$} (B);

% Pの内角
\draw[thick] (44:0.4) arc (44:136:0.4);
\node at (90:0.6) {$\alpha$};

% Aの内角
\draw[thick] (A) ++(-74:0.35) arc (-74:15:0.35);
\path (A) ++(-30:0.55) node {$\beta$};

% Bの内角
\draw[thick] (B) ++(165:0.35) arc (165:254:0.35);
\path (B) ++(210:0.55) node {$\gamma$};
\fill (P) circle (1.5pt) node[below] {$p$};
\end{scope}

%==========================================
% Right Panel: Model M
%==========================================
\begin{scope}[xshift=3.2cm]
\coordinate (P) at (0,0);
\coordinate (A) at (-1.2, 2);
\coordinate (B) at (1.2, 2);
\node[anchor=south west] at (-1.2, 2.5)
{$\wt{M} \supset \wt{\triangle}$};

\draw[thick] (P) -- (A) node[midway, left=3pt] {$r_1$};
\draw[thick] (P) -- (B) node[midway, right=3pt] {$r_2$};
\draw[thick] (A) -- (B) node[midway, above=3pt] {$s$};

% Pの内角
\draw[thick] (59:0.4) arc (59:121:0.4);
\node at (90:0.6) {$\wt{\alpha}$};

% Aの内角
\draw[thick] (A) ++(-59:0.35) arc (-59:0:0.35);
\path (A) ++(-30:0.55) node {$\wt{\beta}$};

% Bの内角
\draw[thick] (B) ++(180:0.35) arc (180:239:0.35);
\path (B) ++(210:0.55) node {$\wt{\gamma}$};
\fill (P) circle (1.5pt) node[below] {$\wt{p}$};
\end{scope}

% 比較の注釈
\draw[->, thick] (-1.2, 1) -- (1.2, 1) node[midway, above] {Comparison};
\node at (0, 0.4) {$\alpha \ge \wt{\alpha}, \ \beta \ge \wt{\beta}, \ \gamma \ge \wt{\gamma}$};
\end{tikzpicture}
\caption{A new type of Toponogov comparison theorem (Theorem \ref{new_TCT}): 
For geodesic triangles with identical 
corresponding side lengths $r_1, r_2, s$, 
the interior angles in $M$ are bounded from 
below by those in the model surface $\wt{M}$. 
The outward bulging of the geodesics 
in $M$ schematically illustrates the geometric 
effect of a lower radial curvature bound, 
rendering the geodesic triangle in $M$ 
a {\em fat triangle} relative to its model counterpart.}
\label{fig:toponogov}
\end{figure}
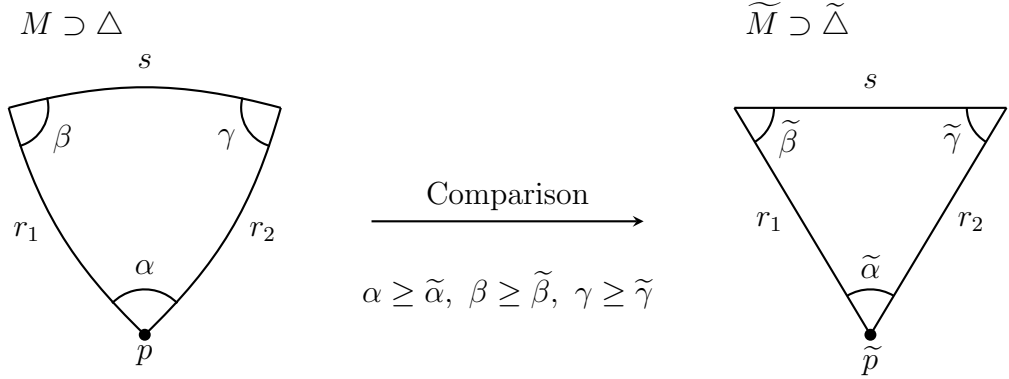
\FloatBarrier

\begin{remark}
We make two remarks concerning Theorem~\ref{new_TCT}:
\begin{enumerate}[{\rm (1)}]
\item If $(\wt{M}, \tilde{p})$ is a von Mangoldt 
surface of revolution, i.e., 
$\wt{K}$ is non-increasing on $[0, \infty)$, 
it follows from \cite[Main Theorem]{T} that 
$\wt{V}(\pi)$ has no pair of cut points. 
If $(\wt{M}, \tilde{p})$ is a Cartan--Hadamard 
surface of revolution, i.e., its Gaussian 
curvature is non-positive on $\wt{M}$, 
it follows from Hadamard's theorem 
(cf.\ \cite[Theorem 3.1 in Chap.\ 7]{dC}) 
that $\wt{V}(\pi)$ has no pair of cut points, 
since $\wt{M}$ is simply connected. 
The assumption on $\wt{V}(\delta_{0})$ 
in Theorem \ref{new_TCT} is therefore 
automatically satisfied if we employ either 
of these two surfaces of revolution 
as the base model $(\wt{M}, \tilde{p})$. 
\item We note that Theorem \ref{new_TCT} was originally stated
in \cite[Theorem 4.12]{KT1} for $\delta_0\in(0,\pi)$.
In the subsequent formulation \cite[Section 3]{KT2}, 
the endpoint case $\delta_0=\pi$ is included 
in view of part~{\rm (1)} above.
\end{enumerate}
\end{remark}

It follows from Theorem~\ref{2025_11_27_thm1.4} that, for the applications below (see Section~\ref{Proof_Cor_end_estimate}), 
we only need the following distance comparison 
for geodesic segments in a Cartan--Hadamard 
model surface of revolution. 
We include a direct proof for completeness.

\begin{lemma}{\rm (Alexandrov's Lemma)}\label{2026_08_08_alexandrov_distance}
Let $(M,p)$ be a complete Riemannian manifold whose radial curvature at $p$ is bounded from 
below by that of a Cartan--Hadamard model 
surface of revolution $(\wt M,\tilde p)$.
Let 
\[
\triangle(pxy)
\quad\text{and}\quad
\triangle(\tilde{p}\tilde{x}\tilde{y})
\]
be corresponding geodesic triangles 
furnished by Theorem~\ref{new_TCT}. 
Let 
\[
\mu:[0,\ell]\to M 
\quad\text{and}\quad
\tilde\mu:[0,\ell]\to\wt M
\]
be the minimizing geodesic segments 
joining $x$ to $y$ and 
$\tilde{x}$ to $\tilde{y}$, respectively, 
parametrized by arc length, where
$\ell:=d(x,y)=\tilde{d} (\tilde{x},\tilde{y})$. 
Then
\[
\tilde{d} (\tilde{p},\tilde{\mu}(s))
\le d(p,\mu(s))
\]
for all $s\in[0,\ell]$.
\end{lemma}

\begin{proof}
The assertion is clear for $s=0$ and $s=\ell$. 
Fix $s\in(0,\ell)$, and set
\[
z:=\mu(s),
\qquad
\tilde{z}:=\tilde{\mu}(s),
\]
and
\[
r:=d(p,z),
\qquad
R:=\tilde{d}(\tilde{p},\tilde{z}).
\]
We shall prove that $R\le r$.

Since $z$ lies in $\mu((0,\ell))$, we have
\begin{equation}\label{2026_08_08_alexandrov_distance_eq2}
\angle(pzx)+\angle(pzy)=\pi.
\end{equation}
Applying Theorem~\ref{new_TCT} to the two geodesic triangles
$\triangle(pxz)$ and $\triangle(pzy)$, we obtain corresponding
geodesic triangles
\[
\triangle(\tilde{p}\tilde{x}_1\tilde{z}_1)
\quad\text{and}\quad
\triangle(\tilde{p}\tilde{z}_2\tilde{y}_2)
\]
in $\wt{M}$ satisfying
\[
\tilde{d}(\tilde{p},\tilde{x}_1)=d(p,x),
\quad
\tilde{d}(\tilde{p},\tilde{z}_1)=r,
\quad
\tilde{d}(\tilde{x}_1,\tilde{z}_1)=s,
\]
and
\[
\tilde{d}(\tilde{p},\tilde{z}_2)=r,
\quad
\tilde{d}(\tilde{p},\tilde{y}_2)=d(p,y),
\quad
\tilde{d}(\tilde{z}_2,\tilde{y}_2)=\ell-s.
\]
In particular,
\begin{equation}\label{2026_08_09_neweq1}
\tilde{d}(\tilde{p},\tilde{z}_1) 
= \tilde{d}(\tilde{p},\tilde{z}_2) 
= r
\end{equation}
and
\begin{equation}\label{2026_08_09_neweq2}
\tilde{d}(\tilde{x},\tilde{y})
=
\ell
=
\tilde{d}(\tilde{x}_1,\tilde{z}_1)
+
\tilde{d}(\tilde{z}_2,\tilde{y}_2).
\end{equation}

Moreover, the angle comparison in
Theorem~\ref{new_TCT} implies 
\[
\angle(pzx)
\ge
\angle(\tilde{p}\tilde{z}_1\tilde{x}_1)
\]
and
\[
\angle(pzy)
\ge
\angle(\tilde{p}\tilde{z}_2\tilde{y}_2).
\]
Hence, by Eq.\,\eqref{2026_08_08_alexandrov_distance_eq2},
\begin{equation}\label{2026_08_08_alexandrov_distance_eq3}
\angle(\tilde{p}\tilde{z}_1\tilde{x}_1)
+
\angle(\tilde{p}\tilde{z}_2\tilde{y}_2)
\le \pi.
\end{equation}

Since $\triangle(pxy)$ and 
$\triangle(\tilde{p}\tilde{x}\tilde{y})$ 
are corresponding geodesic triangles,
\[
\tilde{d}(\tilde{p},\tilde{x})
=
\tilde{d}(\tilde{p},\tilde{x}_1)
\]
and
\[
\tilde{d}(\tilde{p},\tilde{y})
=
\tilde{d}(\tilde{p},\tilde{y}_2).
\]
Therefore, 
Eqs.\,\eqref{2026_08_09_neweq1}, \eqref{2026_08_09_neweq2}, and \eqref{2026_08_08_alexandrov_distance_eq3} verify the hypotheses 
of \cite[Lemma 4.10]{KT1}. 
Thus,
\begin{equation}\label{2026_08_09_neweq4}
\angle(\tilde{p}\tilde{x}_1\tilde{z}_1)
\ge
\angle(\tilde{p}\tilde{x}\tilde{y})
\end{equation}
and
\[
\angle(\tilde{p}\tilde{y}_2\tilde{z}_2)
\ge
\angle(\tilde{p}\tilde{y}\tilde{x}).
\]
Since $\tilde{z}$ lies on $\tilde{\mu}((0,\ell))$, 
we have
\[
\angle(\tilde{p}\tilde{x}\tilde{y})
=
\angle(\tilde{p}\tilde{x}\tilde{z}),
\]
and hence
\begin{equation}\label{2026_08_09_neweq5}
\angle(\tilde{p}\tilde{x}_1\tilde{z}_1)
\ge
\angle(\tilde{p}\tilde{x}\tilde{z}).
\end{equation}

%%%%%%

We now use the Cartan--Hadamard assumption. 
After applying an isometry of the model fixing 
$\tilde{p}$, 
we may assume that $\tilde{x}_1=\tilde{x}$. 
The two geodesic segments $\tilde{x}\tilde{z}_1$ 
and $\tilde{x}\tilde{z}$ have the same length $s$, while Eq.\,\eqref{2026_08_09_neweq5} says that 
the former makes an angle with the segment 
$\tilde{x}\tilde{p}$ no smaller than the latter.

Since $\wt{M}$ is a Cartan--Hadamard surface,
$\exp_{\tilde{x}}:T_{\tilde{x}}\wt{M}\to\wt{M}$ 
is a diffeomorphism. 
For $\omega\in[0,\pi]$, 
let $v_\omega$ be the unit vector making angle 
$\omega$ with the initial direction 
of the segment $\tilde{x}\tilde{p}$, 
and set
\[
\tilde{q}_\omega:=\exp_{\tilde{x}}(s v_\omega),
\qquad
F(\omega):=\tilde{d}(\tilde{p},\tilde{q}_\omega).
\]
Thus the curve
\[
\omega\longmapsto\tilde{q}_\omega
\]
lies on the geodesic circle 
$\{\tilde{a}\in\wt{M}\mid\tilde{d}(\tilde{x},\tilde{a})=s\}$ 
and serves as the terminal transversal curve 
for the variation through the minimizing geodesics from $\tilde{p}$ to $\tilde{q}_\omega$.

Let $\zeta_\omega$ denote the angle at 
$\tilde{q}_\omega$
of the geodesic triangle
$\triangle(\tilde{p}\tilde{x}\tilde{q}_\omega)$. 
By the Gauss lemma, the tangent vector
\[
\frac{\partial\tilde{q}_\omega}{\partial\omega}
\]
to this transversal curve is orthogonal 
to the radial geodesic from 
$\tilde{x}$ to $\tilde{q}_\omega$. 
Hence, by the first variation formula 
for the distance function,
\[
F'(\omega)
=
\left\|
\frac{\partial\tilde{q}_\omega}{\partial\omega}
\right\|
\sin\zeta_\omega
\ge 0
\]
for all $\omega\in(0,\pi)$. 
Hence the opposite-side length of the hinge is nondecreasing with respect to its angle. 
Therefore, Eq.\,\eqref{2026_08_09_neweq5} 
shows that
\[
\tilde{d}(\tilde{p},\tilde{z}_1)
\ge
\tilde{d}(\tilde{p},\tilde{z}).
\]
By the definitions of $r$ and $R$, this is
\[
r\ge R.
\]
Hence
\[
\tilde{d}(\tilde{p},\tilde{\mu}(s))
\le
d(p,\mu(s)).
\]
Since $s\in(0,\ell)$ was arbitrary, 
the proof is complete.
$\qedd$
\end{proof}

Sufficiently large volume growth 
relative to the comparison model 
yields a global rigidity conclusion. 
The volume-growth rigidity theorem stated 
below will provide a key ingredient in the proof 
of our main rigidity theorem 
(Theorem \ref{2025_11_27_thm1.6}).

\begin{theorem}{\rm (\cite[Corollary 3.5]{KT2})}\label{2025_12_13_thm2.7}
Fix $m \in \N$ with $m \ge 2$. 
Define a function $F$ on $[0,\pi]$ by 
\begin{equation}\label{2025_12_13_thm2.7_fn}
F(r) 
:= 
\left(
\int_{0}^{\pi} \sin^{m -2} t\,dt 
\right)^{-1} \int_{0}^{r} \sin^{m -2} t \,dt
\end{equation}
for all $r \in [0,\pi]$. 
Let $(M,p)$ be a connected complete 
non-compact Riemannian $m$-manifold 
whose radial curvature at $p \in M$ is bounded 
from below by $\wt{K}$ of $(\wt{M}, \tilde{p})$, 
and let 
\begin{equation}\label{2025_12_13_thm2.7_const}
\delta (\wt{K}_{-}) 
:= \frac{\pi}{2} 
\exp \left(
\int^{\infty}_{0} t \,\wt{K}_{-} (t) \,dt 
\right),
\end{equation}
where we set $\wt{K}_{-} := \min \{0, \wt{K}\}$. If 
\[
\lim_{t \to \infty} \frac{\Vol B_t (p)}{\Vol B_t^m (\tilde{p})} \ge 1 - F(\delta (\wt{K}_{-}))
\]
then $M$ is diffeomorphic to Euclidean 
$m$-space $\R^m$. 
Here $\Vol B_t (p)$ denotes 
the volume of the open distance ball $B_t(p)$ centered at $p$ with radius $t>0$, 
and $B_t^m (\tilde{p})$ denotes the open 
distance ball centered at 
$\tilde{p}$ with radius $t > 0$ 
in the $m$-dimensional model manifold 
$(\wt{M}^m,\tilde{p})$ of revolution 
determined by the same warping function 
as $(\wt{M},\tilde{p})$.
\end{theorem}

\begin{remark}\label{2026_08_05_rem_critical_points}
Although Theorem~\ref{2025_12_13_thm2.7} 
is stated in \cite[Corollary~3.5]{KT2} 
with only the diffeomorphism conclusion, 
its proof yields the stronger conclusion 
that the distance function 
$d_p(\,\cdot\,):=d(p, \,\cdot\,)$ 
has no critical points on $M\setminus\{p\}$ 
in the sense of Grove--Shiohama. 
Indeed, when
\[
\lim_{t\to\infty}\Vol B_t^m(\tilde{p})=\infty,
\]
the proof of \cite[Theorem~3.4]{KT2} 
shows directly that no point of $M\setminus\{p\}$ 
is a critical point of $d_p$. 
On the other hand, when
\[
\lim_{t\to\infty}\Vol B_t^m(\tilde{p})<\infty,
\]
the proof of \cite[Corollary~3.5]{KT2} 
invokes \cite[Theorem~1.2]{ST_MZ}. 
In this case the warping function 
of the comparison model satisfies
\[
\liminf_{t\to\infty} f(t)=0,
\]
and the proof of \cite[Theorem~1.2]{ST_MZ} 
shows that the cut locus of $p$ is empty. 
Hence $d_p$ has no critical points 
on $M\setminus\{p\}$ in this case as well. 
Therefore, under the assumptions 
of Theorem~\ref{2025_12_13_thm2.7},
\[
\operatorname{Crit}(d_p)\cap(M\setminus\{p\})=\emptyset,
\]
where $\operatorname{Crit}(d_p)$ 
denotes the set of all critical points of $d_p$. 
\end{remark}

\subsection{The geometry of total curvature 
on complete non-compact surfaces}\label{sub2.2}

As discussed in the introduction, 
for the surface $(\Sigma,g^*)$ considered 
in this article, 
the curvature at infinity plays a central role 
in controlling both the topology of $\Sigma$ 
and the behavior of its Busemann functions. 
The notion of curvature at infinity was introduced 
by Shioya \cite{Shioya91} in his study of 
complete non-compact surfaces; 
see also \cite{SST} for a systematic treatment. 
To prepare for these arguments, we recall several results on complete non-compact surfaces that 
relate total curvature and curvature at infinity 
to the geometry of rays and Busemann functions. Throughout this subsection, 
let $(S,h)$ be a connected complete non-compact finitely connected Riemannian $2$-manifold 
admitting total curvature, and denote 
its total curvature by $c(S)$. 

\bigskip

We first recall Shiohama's criterion, 
which converts a lower bound for the total 
curvature into the exhaustion property 
of Busemann functions. 
This result will be applied directly in the proof 
of Theorem \ref{2025_12_01_thm1.7}.

\begin{theorem}{\rm (\cite[Main Theorem]{Shio})}
\label{thm:Shiohama-Busemann}
If $S$ has exactly one end, 
and if $S$ satisfies 
\[
c(S)>2\pi\chi(S)-\pi,
\]
then every Busemann function on $S$ 
is an exhaustion, 
where $\chi(S)$ denotes the Euler characteristic 
of $S$. 
\end{theorem}

\noindent
In the present notation, the curvature at infinity 
of $(S, h)$ is defined by
\[
\lambda_\infty(S):=\lambda_\infty(S, h):=2\pi\chi(S)-c(S).
\]
Shiohama \cite{Shio} also proved that 
if $S$ has exactly one end and
\[
c(S)<2\pi\chi(S)-\pi,
\]
then every Busemann function on $S$ 
is nonexhaustive. 
Equivalently, in terms of the curvature at infinity, 
the value $\pi$ is the sharp threshold between 
the two strict cases. 
In the borderline case
\[
c(S)=2\pi\chi(S)-\pi,
\qquad\text{equivalently}\qquad
\lambda_\infty(S)=\pi,
\]
there exists an example for which some 
Busemann functions are exhaustive 
whereas others are nonexhaustive.

\bigskip

We next recall a relation between the mass 
of rays and the curvature at infinity. 
For each $x\in S$, let $A_x$ denote the set 
of all unit vectors tangent to rays emanating 
from $x$, and let $\mu$ denote 
the Lebesgue measure on the unit circle
\[
\mathbb S_x^1:=\{u\in T_xS\mid \|u\|=1\},
\qquad
\|u\|:=\sqrt{h_x(u,u)}.
\]
The following result, 
established in \cite[Corollary 6.2.2]{SST}, 
will be applied directly to obtain 
Eq.\,\eqref{2025_11_27_thm1.6_eq2}.

\begin{proposition}{\rm (\cite[Corollary 6.2.2]{SST})}
\label{SST_cor6.2.2}
If $S$ is homeomorphic to 
$\R^2$ (such a $S$ is called a Riemannian plane), then 
\[
2\pi-\int_S K_+\,d{\rm vol}
\le
\inf_{x\in S}\mu(A_x)
\le 2\pi - c(S) = \lambda_\infty(S),
\]
where $K_+:=\max\{0,K\}$.
\end{proposition}

The following result gives a more detailed 
description of the asymptotic distribution 
of rays when $\lambda_\infty(S)$ is 
less than $2\pi$. 
In particular, it shows that the mass of 
rays approaches $\lambda_\infty(S)$ uniformly 
outside a compact set and that its averages 
over compact exhaustions converge 
to the same quantity. 
These conclusions will be applied directly 
in the proof of Theorem \ref{2025_12_01_thm1.7}.

\begin{theorem}{\rm (\cite[Theorems 6.2.1 and 6.2.2]{SST})}\label{SST_thms}
Assume that $S$ has exactly one end 
and that $\lambda_\infty(S) < 2\pi$. 
Then the following assertions hold:
\begin{enumerate}[{\rm (a)}]
\item For each $\ve >0$, 
there exists a compact subset 
$D_\ve \subset S$ such that 
\[
|\mu (A_x) - \lambda_\infty (S)| < \ve
\]
for all $x \in S \setminus D_\ve$. 
\item For an increasing exhaustion 
$\{E_i\}_{i\in\N}$ of $S$ by compact subsets, 
we have 
\[
\lim_{i\to\infty}
\frac{
\int_{E_i}\mu(A_x)\,d{\rm vol}
}
{
\int_{E_i}\,d{\rm vol}
}
=
\lambda_\infty(S).
\]
\end{enumerate}
\end{theorem}

In the next section, we combine the radial 
curvature comparison techniques developed 
in Subsection \ref{sub2.1} 
with the total-curvature, ray, 
and Busemann-function results 
recalled above. 
When applied to $(\Sigma,g^*)$, 
these tools provide the geometric input needed 
to derive the topological and asymptotic 
conclusions stated in the main theorems.

%%%%Section 3%%%%%

\section{Proofs of Theorem \ref{2025_11_27_thm1.4}, 
Corollary \ref{2026_08_02_end_estimate}, 
and Theorems \ref{2025_11_27_thm1.6} and 
\ref{2025_12_01_thm1.7}}
\label{sec3}%%%%%%%%%%%%%%%%%%%%%%%%%%%%%%%%%%%%%%%%%%%%%%%%%%%%%%%%%%%%%%%%%%%%%%%%%%%%%%%%%

%%%New Part1:Theorems \ref{2025_11_27_thm1.4}の主張(1)から(4)の証明%%%

Throughout this section, we assume that the complete non-compact Riemannian $2$-manifold $(\Sigma, g)$ 
satisfies the standing assumptions 
\eqref{2025_12_01_whole_assumption} 
on the operator $\cL_g$ in Eq.\,\eqref{2025_12_01_Schrodinger}. 
Moreover, we fix a compact subset 
$C \subset \Sigma$ and a positive 
smooth function $u$ on $\Sigma$
furnished by Theorem \ref{FCmetric}~(1).
Enlarging $C$ if necessary, we may and do assume that $\Int C\neq\emptyset$. 
We then let $g^*=u^2g$ be the Fischer-Colbrie metric on $\Sigma$
and denote by $d^*$ the distance induced by $g^*$. For any $x \in \Sigma$ and $r>0$, $B^*_r(x)$ denotes the open metric ball in $(\Sigma, g^*)$, whose area is given by 
$
\Ar (B^*_r(x)) := \int_{B^*_r(x)}d{\rm vol}^*
$. 
Furthermore, $K^*$ denotes 
the Gaussian curvature of $(\Sigma,g^*)$ 
as in Theorem \ref{FCmetric}~(2-1), 
and $\lambda^*_\infty(\Sigma)$ is the curvature 
at infinity given by Eq.\,\eqref{curvature_at_infinity}. 
We also fix an arbitrary point
\[
p\in\Int C.
\]

\subsection{
Proof of Theorem \ref{2025_11_27_thm1.4}
}

\noindent
\textbf{Proof of Assertion {\rm (1)}.}
Since $C$ is compact and $(\Sigma,g^*)$
is complete, there exists $r_0>0$ such that
\[
C\subset B_{r_0}^*(p),
\]
and the Hopf--Rinow theorem shows that 
the closed ball $\overline{B_{r_0}^*(p)}$ is compact. 
Set
\[
\kappa_0
:=
\min\left\{
0,\,
\min_{x\in\overline{B_{r_0}^*(p)}}K^*(x)
\right\}.
\]
Then $\kappa_0$ is a finite non-positive constant.

Choose a smooth cutoff function
$\eta:[0,\infty)\to[0,1]$ such that
\[
\eta(t)=1\quad\text{for }0\le t\le r_0,
\qquad
\eta(t)=0\quad\text{for }t\ge r_0+1,
\]
and define
\[
\wt K(t):=\kappa_0\eta(t)
\]
for all $t\ge0$. Then $\wt K$ is smooth and non-positive on $[0,\infty)$.

Let $\gamma:[0,a]\to\Sigma$ be any 
unit-speed minimizing geodesic emanating 
from $p$. 
If $0\le t\le\min\{a,r_0\}$, then
\[
K^*(\gamma(t))\ge\kappa_0=\wt K(t).
\]
If $r_0<t\le a$, then
\[
d^*(p,\gamma(t))=t>r_0,
\]
and hence $\gamma(t)\notin C$. 
Therefore, by item {\rm (2-1)}
of Theorem~\ref{FCmetric},
\[
K^*(\gamma(t))\ge0\ge\wt K(t).
\]
Thus, the radial curvature of $(\Sigma,g^*)$ at $p$
is bounded from below by $\wt K$.

Let $m:[0,\infty)\to\R$ be the solution of
\[
m''(t)+\wt K(t)m(t)=0 
\]
with $m(0)=0$ and $m'(0)=1$. 
Since $\wt K\le0$, we have
\[
m''(t)=-\wt K(t)m(t)\ge0
\]
as long as $m(t)\ge0$. 
Hence $m'(t)\ge1$ and consequently 
$m(t)\ge t>0$ for every $t>0$. 
Thus,
\[
(\wt{\Sigma}, \tilde{g})
:=
\bigl([0,\infty)\times_m\Sph^1, 
dt^2+m(t)^2d\theta^2\bigr)
\]
defines a non-compact model surface of revolution with base point
$\tilde{p}$, whose radial curvature function is precisely $\wt K$, where $\Sph^1:=\{v \in \R^2\,|\, \|v\|=1\}$. 

Moreover, since $\wt K(t)=0$ for all $t\ge r_0+1$, we have
\begin{equation}\label{2026_08_02_weighted_curvature}
\int_0^\infty t\,\wt K(t)\,dt
=
\int_0^{r_0+1} t\,\wt K(t)\,dt
>-\infty.
\end{equation}
This proves assertion {\rm (1)}. 

\medskip
\noindent
\textbf{Proof of Assertions {\rm (2)} and {\rm (3)}.}
By Assertion {\rm (1)}, the radial curvature of $(\Sigma,g^*)$ at $p$
is bounded from below by $\wt K$. 
By the Bishop--Gromov-type volume comparison for radial curvature \cite{Mao1}, 
the ratio
\[
\frac{\Ar(B_t^*(p))}{\Ar(B_t(\tilde{p}))}
\]
is non-increasing in $t$. Since this ratio tends to $1$ as $t\downarrow0$, the limit
\[
\lim_{t\to\infty}
\frac{\Ar(B_t^*(p))}{\Ar(B_t(\tilde{p}))}
\]
exists in $[0,1]$.

By items {\rm (2-2)} and {\rm (2-3)} of Theorem
\ref{FCmetric}, $\Sigma$ is finitely connected and
$(\Sigma,g^*)$ admits finite total curvature. It then follows from
the isoperimetric inequality \cite[Theorem 5.2.1]{SST} that
\begin{equation}\label{2025_12_02_proof1_eq3}
\lambda_\infty^*(\Sigma)
=
\lim_{t\to\infty}
\frac{2\Ar(B_t^*(p))}{t^2}.
\end{equation}

On the other hand, by Assertion {\rm (1)}, $\wt K$ is non-positive and compactly supported. In particular, Eq.\,\eqref{2026_08_02_weighted_curvature}
implies 
\[
\int_0^\infty t\,\wt K(t)\,dt>-\infty.
\]
Let $m$ be the warping function of $\wt\Sigma$, so that
$\tilde{g}=dt^2+m(t)^2d\theta^2$ and 
\[
m''(t)+\wt K(t)m(t)=0,
\]
with $m(0)=0$ and $m'(0)=1$. 
Since $\wt K(t)=0$ for all $t\ge r_0+1$, we have
\[
m''(t)=0
\]
for all $t\ge r_0+1$. 
Hence, $m'$ is constant on $[r_0+1,\infty)$, and therefore
\begin{equation}\label{2026_08_02_mprime_limit}
\lim_{t\to\infty}m'(t)<\infty.
\end{equation}
Since $\wt K\le0$, we have $m''\ge0$, 
and hence $m'$ is non-decreasing. 
Thus,
\[
1\le m'(t)\le \lim_{s\to\infty}m'(s)<\infty.
\]
By the generalized l'Hospital theorem (cf.\,\cite[Lemma 5.2.1]{SST}) and Eq.\,\eqref{2026_08_02_mprime_limit},  
we have 
\begin{equation}\label{2026_07_30_total_curvature}
\lim_{t\to\infty}
\frac{\int_0^t m(r)\,dr}{t^2}
=
\frac{1}{2}
\lim_{t\to\infty}\frac{m(t)}{t}
=
\frac12
\lim_{t\to\infty}m'(t)
<\infty.
\end{equation}
We set
\[
\alpha^*
:=
\lim_{t\to\infty}
\frac{\int_0^t m(r)\,dr}{t^2}.
\]
Thus, 
Eq.\,\eqref{2026_07_30_total_curvature} shows
\begin{equation}\label{2026_08_02_alpha_mprime}
2\alpha^*
=
\lim_{t\to\infty}m'(t).
\end{equation}

Since
\[
\Ar(B_t(\tilde{p}))
=
2\pi\int_0^t m(r)\,dr,
\]
Eq.\,\eqref{2025_12_02_proof1_eq3} shows 
\begin{equation}\label{2025_12_02_proof1_eq4}
\lambda_\infty^*(\Sigma)
=
\lim_{t\to\infty}\frac{2\Ar(B_t^*(p))}{t^2}
=
4\pi
\lim_{t\to\infty}
\left\{
\frac{\Ar(B_t^*(p))}{\Ar(B_t(\tilde{p}))}
\frac{\int_0^t m(r)\,dr}{t^2}
\right\}.
\end{equation}
Hence,
\begin{equation}\label{2025_12_02_proof1_eq5}
\lambda_\infty^*(\Sigma)
=
4\alpha^*\pi
\lim_{t\to\infty}
\frac{\Ar(B_t^*(p))}{\Ar(B_t(\tilde{p}))}.
\end{equation}
This proves Assertion {\rm (2)}.

We next prove Assertion {\rm (3)}. Since the area element of
$\wt\Sigma$ is given by
\[
d\wt\Sigma=m(r)\,dr\,d\theta,
\]
and since
\[
\wt K(r)=-\frac{m''(r)}{m(r)},
\]
we obtain
\begin{align*}
c(\wt\Sigma)
&=
\lim_{t\to\infty}
\int_0^{2\pi}\int_0^t
\wt K(r)m(r)\,dr\,d\theta 
=
-2\pi\lim_{t\to\infty}
\int_0^t m''(r)\,dr \\[1mm]
&=
2\pi\left(
1-\lim_{t\to\infty}m'(t)
\right) =
2\pi(1-2\alpha^*).
\end{align*}
In particular, $c(\wt\Sigma)$ is finite. 
Since $\wt K\le0$, we have $c(\wt\Sigma)\le0$. 
Hence
\[
\alpha^*\ge\frac{1}{2}.
\]
Therefore $\alpha^*\in[1/2,\infty)$, 
and Assertion {\rm (3)} follows.

\medskip
\noindent
\textbf{Proof of Assertion {\rm (4)}.}
By Assertion {\rm (1)}, the radial curvature 
of $(\Sigma,g^*)$ at $p$ is bounded from below 
by the radial curvature function $\wt{K}$ 
of the model surface $(\wt{\Sigma},\tilde{p})$. Moreover, 
Eq.\,\eqref{2026_08_02_weighted_curvature} 
implies 
\[
\int_0^\infty (-t\,\wt{K}(t))\,dt<\infty.
\]
Since $\wt{K}\le0$, the model surface $\wt{\Sigma}$ 
is Cartan--Hadamard; 
in particular, the sector condition required 
in \cite[Theorem 5.3]{KT1} is 
automatically satisfied. 
Thus, the setting of that theorem applies. 
More precisely, the argument in the proof of
\cite[Theorem 2.2]{KT1}, which is used in the proof of
\cite[Theorem 5.3]{KT1}, rules out the existence of a sequence
of critical points $\{q_i\}_{i\in\N}$ of $d_p^*$ such that
\[
\lim_{i\to\infty}d_p^*(q_i)=\infty.
\]
Hence, the set of critical points of $d_p^*$ is bounded.
Since $(\Sigma,g^*)$ is complete, the Hopf--Rinow theorem yields a compact subset $\Omega\subset\Sigma$ 
containing $p$ and all critical points of $d_p^*$.

Since $C$ is compact, we can choose $R>0$ sufficiently large so that
\[
C\cup \Omega\subset \ol{B_R^*(p)}.
\]
Thus, $\ol{B_R^*(p)}$ contains both $C$ and all critical points of
$d_p^*$, which proves Assertion {\rm (4)}.

\medskip

This completes the proof of Theorem~\ref{2025_11_27_thm1.4}.$\qedd$

%%%New Part2:Corollary \ref{2026_08_02_end_estimate}の証明%%%

\subsection{Proof of Corollary \ref{2026_08_02_end_estimate}}\label{Proof_Cor_end_estimate}
The following argument is inspired by the proof 
of \cite[Theorem C]{KO} and also corrects a gap in that proof. 
Since the present setting requires several additional steps, we include the proof in full. 
Let $(\wt{\Sigma}, \tilde{g})$ and $\alpha^*$ be as in Theorem \ref{2025_11_27_thm1.4}, 
and let $\cE(\Sigma)$ denote the set of all ends of $\Sigma$. Since $(\Sigma, g^*)$ is complete and noncompact, $\Sigma$ has at least one end. 
If $\#\cE(\Sigma)=1$, then
\[
\#\cE(\Sigma)=1\le 4\alpha^*
\]
because $\alpha^*\ge 1/2$. 
Thus, we may assume henceforth that $\#\cE(\Sigma)\ge 2$. 

For each ${\bf e}\in\cE(\Sigma)$, put
\[
A_{\bf e}
:=
\left\{
v\in(\Sph_p^1)^*
\ \middle|\
\begin{array}{c}
\text{the ray $\gamma_v:[0,\infty)\to(\Sigma,g^*)$ satisfying}\\
\text{$\gamma_v(0)=p$ and $\gamma_v'(0)=v$ belongs to ${\bf e}$}
\end{array}
\right\},
\]
where $(\Sph_p^1)^*:=\{v\in T_p\Sigma\mid \|v\|^*=1\}$.
By a standard limiting argument, each $A_{\bf e}$ is nonempty. 
Indeed, choose a sequence $\{q_i\}$ diverging to infinity
within the end ${\bf e}$, and let $\sigma_i$ be a minimizing
geodesic segment from $p$ to $q_i$. 
Passing to a subsequence, 
the initial vectors $\sigma_i'(0)$ converge 
to a unit vector 
$v\in(\Sph_p^1)^*$, and the corresponding segments converge locally uniformly to a ray $\gamma_v$. 
For every sufficiently large $\varrho>0$, the portion
\[
\sigma_i([\varrho,d^*(p,q_i)])
\]
lies in the component of
$\Sigma\setminus\overline{B_{\varrho}^*(p)}$
representing ${\bf e}$ for all sufficiently large $i$. 
Passing to the limit, we see that $\gamma_v$ also belongs
to the end ${\bf e}$. Hence $v\in A_{\bf e}$.

We claim that, for any two distinct ends
${\bf e}_1,{\bf e}_2\in\cE(\Sigma)$ and any
$v_j\in A_{{\bf e}_j}$, $j=1,2$, one has
\begin{equation}\label{2025_12_02_proof1_eq7}
\angle(v_1,v_2)
\ge
\frac{2\pi^2}{2\pi-c(\wt{\Sigma})}.
\end{equation}
Note that, if $\angle(v_1,v_2)=\pi$, then, since $c(\wt{\Sigma})\le0$,
\[
\frac{2\pi^2}{2\pi-c(\wt{\Sigma})}
\le \pi
=
\angle(v_1,v_2),
\]
and hence Eq.\,\eqref{2025_12_02_proof1_eq7} follows.
We may therefore assume that
\begin{equation}\label{2025_12_02_proof1_eq_angle}
\angle(v_1,v_2)<\pi.
\end{equation}

Let $\gamma_j:=\gamma_{v_j}$, $j=1,2$, be the corresponding rays
emanating from $p$.  Thus
\[
\gamma_j(0)=p,\qquad
\gamma_j'(0)=v_j\in A_{{\bf e}_j},
\qquad
j=1,2.
\]
For each $i\in\N$, set
\[
x_i:=\gamma_1(i),
\qquad
y_i:=\gamma_2(i).
\]
Since ${\bf e}_1\ne{\bf e}_2$, there exists a compact subset
$\cD \subset\Sigma$ such that, for all sufficiently large $i$,
the points $x_i$ and $y_i$ lie in distinct connected components
of $\Sigma\setminus \cD$.
Let $\mu_i$ be a minimizing geodesic segment joining $x_i$ to $y_i$.
Then $\mu_i$ intersects $\cD$.
Choose a point $z_i\in\mu_i\cap \cD$, and parametrize 
$\mu_i$ by arc length so that
$\mu_i:[-a_i,b_i]\to(\Sigma,g^*)$, 
$\mu_i(-a_i)=x_i$, $\mu_i(0)=z_i$, and $\mu_i(b_i)=y_i$, 
where $a_i:=d^*(x_i,z_i)$ and $b_i=d^*(z_i,y_i)$. 
Fix $\rho_0>0$ such that
\[
\cD\subset\overline{B_{\rho_0}^*(p)}.
\]
Since $\angle(v_1,v_2)<\pi$, the minimizing geodesic segment
$\mu_i$ cannot pass through $p$.\footnote{Indeed, suppose $p\in\mu_i$. Then we have $d^*(x_i,y_i)
=d^*(x_i,p)+d^*(p,y_i)$. This implies $\angle(v_1,v_2) = \angle (x_ipy_i)= \pi$. This contradicts $\angle(v_1,v_2) < \pi$.} 
Hence
\begin{equation}\label{2026_07_29_upper}
0<d^*(p,z_i)\le \rho_0.
\end{equation}
Moreover, by the triangle inequality, we have 
\[
a_i\ge d^*(p,x_i)-d^*(p,z_i)\ge i-\rho_0,
\quad
b_i\ge d^*(p,y_i)-d^*(p,z_i)\ge i-\rho_0,
\]
and hence 
\begin{equation}\label{2025_12_02_proof1_eq_lengths_to_infty}
\lim_{i \to \infty}a_i =\infty,
\qquad
\lim_{i \to \infty}b_i =\infty. 
\end{equation} 

Fix $i\in\N$ sufficiently large.
Consider the geodesic triangle $\triangle(px_i y_i)$ 
on $(\Sigma,g^*)$ whose sides are 
$\gamma_1|_{[0,i]}$, $\gamma_2|_{[0,i]}$, and $\mu_i$. 
Since $\wt K\le0$, $\wt\Sigma$ has no cut points. 
Hence Theorem~\ref{new_TCT} is applicable 
with $\delta_0=\pi$ in view of 
Eq.\,\eqref{2025_12_02_proof1_eq_angle}. 
Applying Theorem~\ref{new_TCT} 
to $\triangle(px_i y_i)$, 
we obtain a comparison triangle
$\triangle(\tilde{p}\tilde{x}_i\tilde{y}_i)$\footnote{
Note that $\triangle(\tilde{p}\tilde{x}_i\tilde{y}_i)$ 
is also non-degenerate, 
for $d^*(x_i,y_i) < d^*(p,x_i)+d^*(p,y_i)$. 
In particular, since $\tilde{z}_i\not=\tilde{p}$, 
$\tilde{d}(\tilde{p},\tilde{z}_i)>0$ holds.} 
on $(\wt{\Sigma},\tilde{p})$, 
formed by geodesic segments 
$\tilde{\gamma}_{1,i}$, $\tilde{\gamma}_{2,i}$, 
and $\tilde{\mu}_i$, corresponding, 
respectively, to $\gamma_1|_{[0,i]}$, 
$\gamma_2|_{[0,i]}$, and $\mu_i$, such that
\begin{equation}\label{2025_12_02_proof1_eq_length_1}
\tilde{d}(\tilde{p},\tilde{x}_i)
= d^*(p,x_i)
= i,
\qquad
\tilde{d}(\tilde{p},\tilde{y}_i)
= d^*(p,y_i)
= i,
\end{equation}
\[
\tilde{d}(\tilde{x}_i,\tilde{y}_i)
= d^*(x_i,y_i),
\]
and
\begin{equation}\label{2025_12_02_proof1_eq8}
\angle(x_i p y_i)
\ge
\angle(\tilde{x}_i\tilde{p}\tilde{y}_i),
\end{equation}
\begin{equation}\label{2025_12_02_proof1_eq9}
\angle(p x_i y_i)
\ge
\angle(\tilde{p}\tilde{x}_i\tilde{y}_i),
\qquad
\angle(p y_i x_i)
\ge
\angle(\tilde{p}\tilde{y}_i\tilde{x}_i).
\end{equation}

Let $\tilde{z}_i\in\tilde{\mu}_i$ be the point corresponding to $z_i\in\mu_i$; namely, 
\begin{equation}\label{2026_07_28_new1}
\tilde{d}(\tilde{x}_i,\tilde{z}_i)=a_i,
\qquad
\tilde{d}(\tilde{z}_i,\tilde{y}_i)=b_i.
\end{equation}
In particular, the domain of 
$\tilde{\mu}_i$ is $[-a_i,b_i]$, and 
$\tilde{\mu}_i$ satisfies 
$\tilde{\mu}_i(-a_i)=\tilde{x}_i$, 
$\tilde{\mu}_i(0)=\tilde{z}_i$, 
and $\tilde{\mu}_i(b_i)=\tilde{y}_i$. 
After translating the arc-length parameters,
Lemma \ref{2026_08_08_alexandrov_distance} 
implies 
\[
\tilde{d}(\tilde{p},\tilde{\mu}_i(s))
\le
d^*(p,\mu_i(s))
\]
for every $s\in[-a_i,b_i]$, and hence, 
together with Eq.\,\eqref{2026_07_29_upper}, 
this shows 
\begin{equation}\label{2025_12_02_proof1_eq10}
\tilde d(\tilde{p},\tilde{z}_i)
\le 
d^*(p,z_i)
\le \rho_0.
\end{equation}
Thus, for any sufficiently large $i$, 
we have
\[
\tilde{\mu}_i([-a_i,b_i])
\cap
\ol{B_{\rho_0}(\tilde{p})}
\ne
\emptyset.
\]
Since $\tilde{\mu}_i(0)=\tilde{z}_i$ 
and $\tilde{d}(\tilde{p},\tilde{z}_i)\le \rho_0$ 
by Eq.\,\eqref{2025_12_02_proof1_eq10}, 
for every $T>0$ and every sufficiently 
large $i$ satisfying $[-T,T]\subset[-a_i,b_i]$, 
we have
\[
\tilde{\mu}_i([-T,T])
\subset
\ol{B_{\rho_0+T}(\tilde{p})}.
\]
Since $(\wt{\Sigma},\tilde{g})$ is complete, 
the Hopf--Rinow theorem implies that 
the closed ball $\ol{B_{\rho_0+T}(\tilde{p})}$
is compact. 
By 
Eq.\,\eqref{2025_12_02_proof1_eq_lengths_to_infty},
we may choose a subsequence 
$\{\tilde\mu_{i_j}\}_{j\in\N}$ of
$\{\tilde\mu_i\}_{i\in\N}$ 
such that
\[
a_{i_j}\ge j,
\qquad
b_{i_j}\ge j
\]
for every $j\in\N$. 
Applying the Arzel\`a--Ascoli theorem 
and a diagonal argument to 
$\{\tilde\mu_{i_j}\}_{j\in\N}$, 
we obtain a locally uniformly convergent subsequence
\[
\{\tilde\mu_{i_k}\}_{k\in\N}:=\{\tilde\mu_{i_{j_k}}\}_{k\in\N}
\]
whose limit is a complete geodesic
\[
\tilde\mu_\infty:\R\to\wt{\Sigma}.
\]
Writing $a_{i_k}:=a_{i_{j_k}}$ and $b_{i_k}:=b_{i_{j_k}}$, we have $a_{i_k}\ge k$ and $b_{i_k}\ge k$, 
since $j_k\ge k$ for every $k\in\N$. 

We finally verify that $\tilde\mu_\infty$ is a line.
Fix $s,t\in\R$. 
Choose $k$ sufficiently large 
so that $s,t\in[-k,k]$. 
Since $[-k,k]\subset [-a_{i_k},b_{i_k}]$, 
the restriction of $\tilde\mu_{i_k}$ 
to $[-k,k]$ is a minimizing geodesic segment.
Hence, we have 
$
\tilde d\bigl(
\tilde\mu_{i_k}(s),
\tilde\mu_{i_k}(t)
\bigr)
=
|s-t|$. Letting $k\to\infty$, we obtain
\[
\tilde d\bigl(
\tilde\mu_\infty(s),
\tilde\mu_\infty(t)
\bigr)
=
|s-t|.
\]
Thus $\tilde\mu_\infty$ is minimizing 
on every compact interval, 
and hence $\tilde\mu_\infty$ is a line.

%%%%

Passing to a further subsequence if necessary, 
we may assume, 
by the Arzel\`a--Ascoli theorem 
and Eq.\,\eqref{2025_12_02_proof1_eq_length_1}, 
that the sequences 
$\{\tilde\gamma_{j,i_k}\}_{k\in\N}$, $j=1,2$, 
converge locally uniformly to 
rays 
$\tilde\gamma_{j,\infty}:[0,\infty)\to\wt{\Sigma}$, $j=1,2$. 
Set 
\[
\omega_p:=\angle(v_1,v_2)=
\angle(x_{i_k}p y_{i_k}), \quad 
\tilde\omega_{\tilde p}
:=
\angle\bigl(
\tilde\gamma_{1,\infty}'(0),
\tilde\gamma_{2,\infty}'(0)
\bigr)
=
\lim_{k\to\infty}
\angle(\tilde x_{i_k}\tilde p\tilde y_{i_k}).
\]
By Eq.\,\eqref{2025_12_02_proof1_eq8}, we have
\begin{equation}\label{2025_12_02_proof1_eq12}
\omega_p\ge\tilde\omega_{\tilde p}.
\end{equation}
On the other hand, 
since $x_{i_k}$ and $y_{i_k}$ diverge to infinity 
along the fixed rays $\gamma_1$ and 
$\gamma_2$, respectively, while each 
minimizing segment $\mu_{i_k}$ intersects 
the fixed compact set $\cD$, 
the generalized first variation formula of 
Itoh--Tanaka \cite[Lemma 2.1]{IT}, applied at the two endpoints, shows 
\[
\lim_{k\to\infty}\angle(p x_{i_k} y_{i_k})=0,
\qquad
\lim_{k\to\infty}\angle(p y_{i_k} x_{i_k})=0.
\]
(This is the classical argument of Cohn-Vossen; see also \cite{CV2}.) 
Consequently, Eq.\,\eqref{2025_12_02_proof1_eq9} implies
\begin{equation}\label{2025_12_02_proof1_eq13}
\lim_{k\to\infty}
\angle(\tilde p\tilde x_{i_k}\tilde y_{i_k})=0,
\qquad
\lim_{k\to\infty}
\angle(\tilde p\tilde y_{i_k}\tilde x_{i_k})=0.
\end{equation}

By Eq.\,\eqref{2025_12_02_proof1_eq12} and
Eq.\,\eqref{2025_12_02_proof1_eq_angle}, we have
\[
\tilde\omega_{\tilde p}
\le \omega_p<\pi.
\]
Since $\wt\Sigma$ is Cartan--Hadamard, 
the closed sector bounded by the two meridians 
is geodesically convex. 
Hence each $\tilde\mu_{i_k}$ is contained 
in the closed sector bounded by 
$\tilde\gamma_{1,i_k}$ and 
$\tilde\gamma_{2,i_k}$. 
Passing to the limit, we see that 
$\tilde\mu_\infty$ is contained 
in the closed sector bounded by 
the two limiting meridians 
$\tilde\gamma_{1,\infty}$ and 
$\tilde\gamma_{2,\infty}$. 
Consequently, $\tilde\mu_\infty$ cannot pass through $\tilde p$; indeed, a line through $\tilde p$ has two opposite tangent directions and therefore cannot be 
contained in a sector of opening angle strictly less than $\pi$. 

Let $\wt X\subset\wt\Sigma$ be the sector of opening angle
$\tilde\omega_{\tilde p}$ bounded 
by the limiting meridians $\tilde\gamma_{1,\infty}([0,\infty))$ 
and $\tilde\gamma_{2,\infty}([0,\infty))$ 
and containing $\tilde\mu_\infty$. 
Let $\wt Y\subset\wt X$ be the portion of the half-plane bounded by $\tilde\mu_\infty$ that does not contain 
$\tilde p$ and lies in $\wt X$, 
and put
\[
\wt Z:=\ol{\wt X\setminus\wt Y}.
\]
Thus, $\wt Z$ is the domain bounded by the two 
limiting meridians and $\tilde\mu_\infty$ that contains 
$\tilde p$. (See Figure \ref{fig:domains-XYZ} below.)

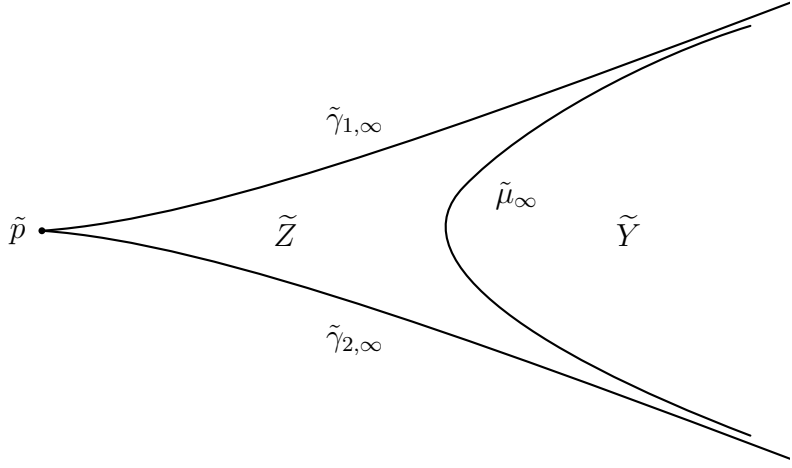
\begin{figure}[htbp]
\centering

\begin{tikzpicture}
[
x=1.15cm,
y=1.05cm,
line cap=round,
line join=round
]

\begin{scope}
\clip (-0.7,-3.15) rectangle (8.6,3.15);

% Base point

\coordinate (p) at (0,0);

% Two limiting meridians

\draw[thick]
(p) .. controls (2.0,0.15) and (5.2,1.45) .. (10.0,3.45);
\draw[thick]
(p) .. controls (2.0,-0.15) and (5.2,-1.45) .. (10.0,-3.45);

% The complete geodesic $\widetilde\mu_\infty$.
% It remains strictly inside $\widetilde X$ and approaches
% the two limiting meridians toward the right.

\draw[thick]
(8.15,2.58) .. controls (7.00,2.20) 
and (5.55,1.35) .. (4.85,0.55) .. controls (4.10,-0.30) 
and (5.45,-1.45) .. (8.15,-2.58);
\end{scope}

% Base point and its label

\coordinate (p) at (0,0); 
\fill (p) circle (1.3pt);
\node[left=2pt] at (p) {$\tilde p$};

% Labels of the limiting meridians

\node[above] at (3.6,1.05)
{$\tilde\gamma_{1,\infty}$};
\node[below] at (3.6,-1.05)
{$\tilde\gamma_{2,\infty}$};

% Label of the complete geodesic

\node[right=3pt] at (5.0,0.45)
{$\tilde\mu_\infty$};

% Labels of the two subdomains

\node at (2.8,0)
{$\widetilde Z$};
\node at (6.75,0)
{$\widetilde Y$};

\end{tikzpicture}

\caption{
A schematic illustration of the domains 
$\widetilde Y$ and 
$\widetilde Z=\overline{\widetilde X\setminus\widetilde Y}$, 
so that 
$\widetilde X=\widetilde Y\cup\widetilde Z$. 
The line $\widetilde\mu_\infty$ separates 
these two domains inside the sector bounded 
by the limiting meridians 
$\widetilde\gamma_{1,\infty}$ 
and $\widetilde\gamma_{2,\infty}$.
}
\label{fig:domains-XYZ}
\end{figure}
\FloatBarrier

%%%%%%

We denote by $c(\wt X)$, $c(\wt Y)$, and $c(\wt Z)$ 
the total curvatures of $\wt X$, $\wt Y$, and $\wt Z$, respectively. 
Set 
\[
\wt\triangle_k
:=
\triangle(\tilde p\tilde x_{i_k}\tilde y_{i_k}).
\]
The Gauss--Bonnet theorem implies 
\begin{equation}\label{2025_12_02_proof1_eq14}
c(\wt\triangle_k)
=
\angle(\tilde x_{i_k}\tilde p\tilde y_{i_k})
+\angle(\tilde p\tilde x_{i_k}\tilde y_{i_k})
+\angle(\tilde p\tilde y_{i_k}\tilde x_{i_k})
-\pi,
\end{equation}
where $c(\wt\triangle_k)$ denotes the total curvature 
of $\wt\triangle_k$. 
Let $\chi_k$ and $\chi_Z$ denote the characteristic 
functions of $\widetilde\triangle_k$ and $\widetilde Z$, respectively. 
The locally uniform convergence of the three sides, together with the choice of these domains, implies that $\chi_k\to\chi_Z$ almost everywhere on $\widetilde\Sigma$. 
Indeed, every point outside the three limiting boundary
geodesics has a positive distance from each of these
geodesics; hence, on a sufficiently small neighborhood of
such a point, the domains $\widetilde\triangle_k$ and
$\widetilde Z$ agree for all sufficiently large $k$.
Since the limiting boundary has measure zero, this
justifies the almost-everywhere convergence above. 
Since $\widetilde K\le0$ and
\[
\int_{\widetilde\Sigma}|\widetilde K|\,d\wt{\Sigma}<\infty,
\]
the dominated convergence theorem shows 
\begin{equation}\label{2026_07_30_convergence}
\lim_{k\to\infty}c(\widetilde\triangle_k)
=
\lim_{k\to\infty}
\int_{\widetilde\Sigma}\chi_k\widetilde K\,d\wt{\Sigma}
=
\int_{\widetilde\Sigma}\chi_Z\widetilde K\,d\wt{\Sigma}
=c(\widetilde Z).
\end{equation}
By Eqs.\,\eqref{2025_12_02_proof1_eq13} and \eqref{2026_07_30_convergence}, 
taking the limit in Eq.\,\eqref{2025_12_02_proof1_eq14}, 
we obtain
\begin{equation}\label{2025_12_02_proof1_eq15}
c(\wt Z)
=
\lim_{k\to\infty}c(\wt\triangle_k)
=
\tilde\omega_{\tilde p}-\pi.
\end{equation}
Since $\wt K\le0$, we have $c(\wt Y)\le0$. 
Moreover, by the rotational symmetry of $\wt\Sigma$, we have
\[
c(\wt X)
=
\frac{\tilde\omega_{\tilde p}}{2\pi}
c(\wt\Sigma).
\]
Thus, Eq.\,\eqref{2025_12_02_proof1_eq15} yields
\begin{equation}\label{2025_12_02_proof1_eq16}
\frac{\tilde\omega_{\tilde p}}{2\pi}c(\wt\Sigma)
=
c(\wt X)
=
c(\wt Y)+c(\wt Z)
\le
c(\wt Z)
=
\tilde\omega_{\tilde p}-\pi.
\end{equation}
Since $2\pi-c(\wt\Sigma)>0$, 
Eqs.\,\eqref{2025_12_02_proof1_eq12} and \eqref{2025_12_02_proof1_eq16} give 
\begin{equation}\label{2025_12_02_proof1_eq17}
\frac{2\pi^2}{2\pi-c(\wt\Sigma)}
\le
\tilde\omega_{\tilde p}
\le
\omega_p
=
\angle(v_1,v_2).
\end{equation}
This proves the claim \eqref{2025_12_02_proof1_eq7}.

For any nonempty subsets $A,B\subset(\Sph_p^1)^*$, 
define
\[
\dist(A,B)
:=
\inf\{\angle(v,w)\mid v\in A,\ w\in B\}.
\]
By setting
\[
\lambda
:=
\frac{\pi^2}{2\pi-c(\wt\Sigma)},
\]
it follows from Eq.\,\eqref{2025_12_02_proof1_eq17} that, 
for any distinct ends ${\bf e},{\bf f}\in\cE(\Sigma)$,
\begin{equation}\label{2025_12_02_proof1_eq18}
\dist(A_{\bf e},A_{\bf f})\ge 2\lambda.
\end{equation}
For each ${\bf e}\in\cE(\Sigma)$, choose
$v_{\bf e}\in A_{\bf e}$. Then
Eq.\,\eqref{2025_12_02_proof1_eq18} implies that the open balls
\[
\left\{\B_\lambda(v_{\bf e})\right\}_{{\bf e}\in\cE(\Sigma)}
\]
are mutually disjoint in $(\Sph_p^1)^*$. 
Let $\mathcal F$ be any finite subset of $\cE(\Sigma)$. 
Since $c(\wt\Sigma)\le0$, we have $\lambda\le\pi/2$.
Hence each $\B_\lambda(v_{\bf e})$ 
has length $2\lambda$. 
The packing argument therefore shows 
\[
2\lambda\,\#\mathcal F\le2\pi.
\]
Thus
\[
\#\mathcal F\le\frac{\pi}{\lambda}
\]
for every finite subset $\mathcal F\subset\cE(\Sigma)$. Consequently,
$\cE(\Sigma)$ itself is finite and
\[
\#\cE(\Sigma)
\le
\frac{\pi}{\lambda}
=
2-\frac{c(\wt\Sigma)}{\pi}
=
4\alpha^*.
\]
This proves the corollary. $\qedd$

%%%%%%%%%%%%%%%%%%%%%%%%%%%%%%%%%%%%
%%%New Part3:Theorem \ref{2025_11_27_thm1.6}の証明%%%
%%%%%%%%%%%%%%%%%%%%%%%%%%%%%%%%%%%%

\subsection{Proof of Theorem \ref{2025_11_27_thm1.6}}

Let $F$ be the function defined by
Eq.\,\eqref{2025_12_13_thm2.7_fn}.
Since $\dim\Sigma=2$, we have $m=2$ in the notation of
Theorem~\ref{2025_12_13_thm2.7}, and hence
\[
F(r)=\frac{r}{\pi}
\]
for every $r\in[0,\pi]$. 
Moreover, since $\wt K\le0$, we have 
$\wt K_-=\min\{0,\wt K\}=\wt K$, and hence
$\delta(\wt K_-)=\delta(\wt K)$, where $\delta(\wt K_-)$ is defined 
by Eq.\,\eqref{2025_12_13_thm2.7_const}.

It follows from Theorem \ref{2025_11_27_thm1.4} that
the radial curvature of $(\Sigma,g^*)$ at $p$
is bounded from below by $\wt K$, and that
\begin{equation}\label{2025_11_27_thm1.6_proof_eq1}
\lim_{t\to\infty}
\frac{\Ar(B_t^*(p))}{\Ar(B_t(\tilde p))}
=
\frac{\lambda_\infty^*(\Sigma)}{4\alpha^*\pi}.
\end{equation}
Since
\[
F(\delta(\wt K_-))
=
F(\delta(\wt K))
=
\frac{1}{2}
\exp\left(
\int_0^\infty t\wt K(t)\,dt
\right),
\]
Eqs.\,\eqref{2025_11_27_thm1.6_eq1} and
\eqref{2025_11_27_thm1.6_proof_eq1} imply 
\begin{align*}
\lim_{t\to\infty}
\frac{\Ar(B_t^*(p))}{\Ar(B_t(\tilde p))}
&=
\frac{\lambda_\infty^*(\Sigma)}{4\alpha^*\pi}
\\[1mm]
&\ge
\frac12
\left(
2-
\exp\left(
\int_0^\infty t\wt K(t)\,dt
\right)
\right)
\\[1mm]
&=
1-F(\delta(\wt K_-)).
\end{align*}
Thus the hypothesis of Theorem~\ref{2025_12_13_thm2.7}
is satisfied. 
By Remark~\ref{2026_08_05_rem_critical_points},
$d_p^*$ has no critical points 
on $\Sigma\setminus\{p\}$. 
In particular, $\Sigma$ is diffeomorphic to $\R^2$. 
Hence, in view of the standing assumptions 
on $(\Sigma,g^*)$, 
$\Sigma$ is a Riemannian plane 
in the sense of \cite{SST}. 
Therefore, 
Proposition~\ref{SST_cor6.2.2} applies 
to $(\Sigma,g^*)$, and 
Eq.\,\eqref{2025_11_27_thm1.6_eq2} follows.
$\qedd$

%%%%%%%%%%%%%%%%%%%%%%%%%%%%%%%%%%%%
%%%New Part4:Theorem \ref{2025_12_01_thm1.7}の証明%%%
%%%%%%%%%%%%%%%%%%%%%%%%%%%%%%%%%%%%

\subsection{Proof of Theorem \ref{2025_12_01_thm1.7}}

Since $\lambda_\infty^*(\Sigma)>0$, the area growth formula
\eqref{2025_11_27_thm1.4_eq1} implies that
\[
\lim_{t\to\infty}
\frac{\Ar(B_t^*(p))}
{\Ar(B_t(\tilde p))}
>0.
\]
Define
\[
\beta^*
:=
\alpha^*
\lim_{t\to\infty}
\frac{\Ar(B_t^*(p))}
{\Ar(B_t(\tilde p))}.
\]
Then $\beta^*>0$. By the definition of the curvature at infinity
and Eq.\,\eqref{2025_11_27_thm1.4_eq1},
\begin{equation}\label{2025_12_01_thm1.7_proof_eq1}
2\pi\chi(\Sigma)-c^*(\Sigma)
=
\lambda_\infty^*(\Sigma)
=
4\pi\beta^*.
\end{equation}

We prove the two assertions separately.

\medskip
\noindent
\textbf{Proof of Assertion {\rm (1)}.}
Assume that $\beta^*\in(0,1/4)$. Then
Eq.\,\eqref{2025_12_01_thm1.7_proof_eq1} implies that
\[
c^*(\Sigma)
=
2\pi\chi(\Sigma)-4\pi\beta^*
>
2\pi\chi(\Sigma)-\pi.
\]
Therefore, Theorem \ref{thm:Shiohama-Busemann} shows 
that every Busemann function on $(\Sigma,g^*)$ is 
an exhaustion.

Moreover, Eq.\,\eqref{2025_12_01_thm1.7_proof_eq1} implies 
\[
\lambda_\infty^*(\Sigma)
=
4\pi\beta^*
<
\pi
<
2\pi.
\]
Since $\Sigma$ has exactly one end, all the assumptions of
Theorem \ref{SST_thms} are satisfied. Therefore,
Eqs.\,\eqref{2026_07_31_thm1.7_new1}
and \eqref{2026_07_31_thm1.7_new2}
follow from assertions {\rm (a)} and {\rm (b)} of
Theorem \ref{SST_thms}, respectively.

\medskip
\noindent
\textbf{Proof of Assertion {\rm (2)}.} 
Assume that
\[
\beta^*
\in
\left(0,\frac{\chi(\Sigma)}{2}-\frac{1}{4}\right).
\]
Then we have
\begin{equation}\label{2026_08_01_neweq_1}
\frac{1}{2}<\chi(\Sigma).
\end{equation}
On the other hand, since $\Sigma$ has exactly one end and is of
finite topological type, there exist a connected compact surface
$N$ and a point $q\in N$ such that
\[
\Sigma\cong N\setminus\{q\}.
\]
Consequently,
\begin{equation}\label{2026_08_01_neweq_2}
\chi(\Sigma)=\chi(N)-1\le1,
\end{equation}
since $\chi(N)\le2$. Since $\chi(\Sigma)$ is an integer,
Eqs.\,\eqref{2026_08_01_neweq_1} and
\eqref{2026_08_01_neweq_2} yield
\[
\chi(\Sigma)=1.
\]
Thus $\chi(N)=2$. By the classification theorem for compact
surfaces, $N$ is homeomorphic to the $2$-sphere $\Sph^2$.
Hence
\[
\Sigma
\cong
\Sph^2\setminus\{q\}
\cong
\R^2.
\]
Thus, in view of the standing assumptions on $(\Sigma,g^*)$,
$\Sigma$ is a Riemannian plane in the sense of \cite{SST}.
Therefore, Proposition~\ref{SST_cor6.2.2} applies to
$(\Sigma,g^*)$, and Eq.\,\eqref{2025_11_27_thm1.6_eq2}
follows. 

Moreover, since $\chi(\Sigma)=1$, the assumption in {\rm (2)}
reduces to
\[
\beta^*\in(0,1/4).
\]
It therefore follows from Assertion {\rm (1)} that every Busemann
function on $(\Sigma,g^*)$ is an exhaustion. This completes the
proof. $\qedd$

\begin{flushleft}
H.\,Harumoto\\
{\small e-mail: 
{\tt  pfai1uqx@s.okayama-u.ac.jp}}

\medskip

K.\,Kondo\\Department of Mathematics, Faculty of Science, Okayama University\\Okayama City, Okayama Pref. 700-0082, Japan\\{\small e-mail: {\tt keikondo@okayama-u.ac.jp}}

\end{flushleft}

\end{document}